\documentclass{amsart}
\usepackage{amsmath,amssymb,amsthm,geometry,graphics,color,mdwlist,mathrsfs}
\usepackage[all]{xy}

\DeclareMathOperator{\image}{Im}
\renewcommand{\Im}{\image}
\DeclareMathOperator{\Ker}{Ker}

\DeclareMathOperator{\Hom}{Hom}
\DeclareMathOperator{\sHom}{\underline{\Hom}}

\DeclareMathOperator{\RHom}{\mathrm{R}\!\Hom}

\DeclareMathOperator{\End}{End}
\DeclareMathOperator{\sEnd}{\underline{\End}}

\def\lotimes{\otimes^\mathrm{L}}

\DeclareMathOperator{\RH}{\mathrm{H}}

\DeclareMathOperator{\Tor}{Tor}
\DeclareMathOperator{\Ext}{Ext}

\DeclareMathOperator{\pd}{proj.dim}

\DeclareMathOperator{\gd}{gl.dim}

\DeclareMathOperator{\md}{mod}
\renewcommand{\mod}{\md}
\DeclareMathOperator{\Mod}{Mod}

\DeclareMathOperator{\proj}{proj}
\DeclareMathOperator{\inj}{inj}

\DeclareMathOperator{\refl}{ref}

\DeclareMathOperator{\CM}{CM}
\DeclareMathOperator{\sCM}{\underline{\CM}}

\DeclareMathOperator{\coh}{coh}

\DeclareMathOperator{\per}{per}

\DeclareMathOperator{\add}{add}

\DeclareMathOperator{\Sym}{Sym}
\DeclareMathOperator{\Spec}{Spec}

\DeclareMathOperator{\Proj}{Proj}

\DeclareMathOperator{\depth}{depth}

\DeclareMathOperator{\ch}{char}

\def\seg{\#}
\def\rD{\mathrm{D}}

\def\rC{\mathrm{C}}

\def\Db{\mathrm{D^b}}
\def\Kb{\mathrm{K^b}}

\newcommand\recollement[3]{\xymatrix{{#1}\ar[r]&{#2}\ar[r]\ar@/_7pt/[l]\ar@/^7pt/[l]&{#3}\ar@/_7pt/[l]\ar@/^7pt/[l] }}
\def\A{\mathscr{A}}

\def\E{\mathscr{E}}

\def\M{\mathscr{M}}

\def\P{\mathscr{P}}

\def\T{\mathscr{T}}
\def\U{\mathscr{U}}
\def\V{\mathscr{V}}

\def\Ga{\Gamma}

\def\La{\Lambda}

\def\Om{\Omega}

\def\al{\alpha}

\def\la{\lambda}

\def\om{\omega}
\def\NN{\mathbb{N}}
\def\Z{\mathbb{Z}}

\def\m{\mathfrak{m}}
\def\n{\mathfrak{n}}

\def\op{\mathrm{op}}

\def\pprime{{\prime\prime}}

\def\rsimeq{\rotatebox{-90}{$\simeq$}}
\def\xsimeq{\xrightarrow{\simeq}}

\newtheorem{Thm}{Theorem}[section]
\newtheorem{Lem}[Thm]{Lemma}
\newtheorem{Prop}[Thm]{Proposition}
\newtheorem{Cor}[Thm]{Corollary}

\newtheorem{Prop-Def}[Thm]{Proposition-Definition}
\newtheorem{Thm-Def}[Thm]{Theorem-Definition}

\theoremstyle{definition}
\newtheorem{Def}[Thm]{Definition}
\newtheorem{Ex}[Thm]{Example}

\newtheorem{Setup}[Thm]{Setting}

\theoremstyle{remark}
\newtheorem{Rem}[Thm]{Remark}

\newcounter{step}

\def\disoplus{\displaystyle\bigoplus}

\def\Oplus{\displaystyle\bigoplus}

\def\Prod{\displaystyle\prod}

\def\Sum{\displaystyle\sum}

\def\Wedge{\textstyle\bigwedge}
\makeatletter

\@addtoreset{equation}{section}
\makeatother

\def\Diff{\Omega}
\def\seg{\#}

\title[NCCRs for Segre products and CM rings of hereditary representation type]{Non-commutative resolutions for Segre products and \\ Cohen-Macaulay rings of hereditary representation type}
\author{Norihiro Hanihara}
\thanks{This work is supported by JSPS KAKENHI Grant Number JP22J00649}
\subjclass[2020]{13C14, 14A22, 16E35, 16G60, 16S38, 18G80}
\keywords{Hereditary representation type, Strictly hereditary representation type, Cohen-Macaulay ring, hereditary algebra, Segre product, non-commutative crepant resolution, CT module, extended numerical semigroup ring, rigid module}
\address{Kavli Institute for the Physics and Mathematics of the Universe (WPI),The University of Tokyo Institutes for Advanced Study, The University of Tokyo, Kashiwa, Chiba 277-8583, Japan}
\email{norihiro.hanihara@ipmu.jp}

\begin{document}
\begin{abstract}
We study commutative Cohen-Macaulay rings whose Cohen-Macaulay representation theory are controlled by representations of quivers, which we call hereditary representation type. Based on tilting theory and cluster tilting theory, we construct some commutative Cohen-Macaulay rings of hereditary representation type.
First we give a general existence theorem of cluster tilting module or non-commutative crepant resolutions on the Segre product of two commutative Gorenstein rings whenever each factor has such an object. As an application we obtain three examples of Gorenstein rings of hereditary representation type coming from Segre products of polynomial rings.
Next we introduce extended numerical semigroup rings which generalize numerical semigroup rings and form a class of one-dimensional Cohen-Macaulay non-domains, and among them we provide one family of Gorenstein rings of hereditary representation type.
Furthermore, we discuss a $4$-dimensional non-Gorenstein Cohen-Macaulay ring whose representations are still controlled by a finite dimensional hereditary algebra. We show that it has a unique $2$-cluster tilting object, and give a complete classification of rigid Cohen-Macaulay modules, which turns out to be only finitely many.
\end{abstract}

\maketitle
\setcounter{tocdepth}{1}
\tableofcontents
\section{Introduction}
\subsection{Hereditary representation types}
The study of Cohen-Macaulay modules over Cohen-Macaulay rings \cite{Yo90,LW} is one of the most important subjects in commutative algebra.
Given a commutative Cohen-Macaulay ring $R$, one aims at understanding the category $\CM R$ of (maximal) Cohen-Macaulay modules over $R$. 
The most fundamental case where one has a complete understanding of the category is when $R$ is of {\it finite representation type}. Recall that it means there are only finitely many indecomposable Cohen-Macaulay modules up to isomorphism. 
%
Such commutative rings, especially the simple singularities in dimension $2$, are classical subject which are well-studied from many perspectives including representation theory \cite{BGS,Au86} and the McKay correspondence \cite{AV}, but in general it is considered as hopeless to understand the whole category $\CM R$.

The aim of this paper is to develop the study of Cohen-Macaulay representations which can be seen as ``next simple'' from finite representation type, from the point of view of tilting theory and cluster tilting theory. We introduce the notion of {\it hereditary representation type}, construct some examples of such rings, and apply it to give some classification results.

Let us explain how this notion is motivated. Let $k$ be a field, and let $R$ be a commutative Noetherian $k$-algebra which we assume to be Gorenstein. Then the stable category $\sCM R$ of Cohen-Macaulay modules is triangulated, which is also canonically equivalent to the {\it singularity category} $\Db(\mod R)/\Kb(\proj R)$ of Buchweitz \cite{Bu86} and Orlov \cite{Or04}. It has been actively studied from many perspectives such as representation theory of rings and groups, algebraic geometry, mirror symmetry, and so on. 

Beyond finite representation type an effective way of studying Cohen-Macaulay representations is tilting theory. It gives equivalences among various triangulated categories arising from different contexts, 
which provides a mutual understanding of these categories, and it is therefore an important subject in representation theory.
Extensive studies on tilting theory for singularity categories e.g.\,\cite{AIR,BIY,FU,haI,HeI,HIMO,HO,IT,KST1,KMV,KRac,KLM,Na,U1}, also a survey article \cite{Iy18}, reveal that we often have a triangle equivalence 
\[ \sCM R\simeq \rC_n(A) \]
for some finite dimensional algebra $A$, realizing the triangulated category $\sCM R$ as a canonical form of such. Here the right-hand-side is the {\it $n$-cluster category} \cite{BMRRT,Am09} of $A$ which is obtained as (the triangulated hull of) an orbit category of the derived category $\Db(\mod A)$.

We intend to measure the complexity the representation theory of $R$ in terms of the finite dimensional algebra $A$. The most basic class of such algebras is formed by {\it hereditary algebras}; recall that it means the algebra has global dimension at most $1$. Over an algebraically closed field, every such algebra is Morita equivalent to the path algebra of a quiver, and the modules are nothing but quiver representations. They are in this way fundamental objects in representation theory, which suggests the following definition.
\begin{Def}[{Definition \ref{hered}}]\label{ihered}
Let $R$ be a commutative complete Gorenstein local ring. We say that $R$ is of {\it hereditary representation type} if there exists a finite dimensional hereditary algebra $H$ and a triangle autoequivalence $F$ of $\Db(\mod H)$ such that there is a triangle equivalence
\[ \sCM R\simeq\Db(\mod H)/F. \]
\end{Def}
As a first example of Gorenstein ring of hereditary representation type, one can state the structure theorem of finite triangulated categories \cite{Am07} as follows.
\begin{Prop}[{Proposition \ref{finite}}]
Gorenstein rings of finite representation type are of hereditary representation type. Moreover, one can take the hereditary algebra $H$ to be the path algebra of a Dynkin quiver.
\end{Prop}
Among many perspectives in which hereditary algebras are distinguished, the first but essential property is that hereditary algebras have concise structure of their derived categories; it is well-known that every object is isomorphic to its cohomology, thus the structure of the derived category $\Db(\mod H)$ is very close to that of the module category. 
As a first consequence of a Gorenstein ring being hereditary representation type, we translate a result of Keller on derived orbit categories of hereditary algebras \cite{Ke05} into gradability of Cohen-Macaulay modules (Theorem \ref{gradable}).

\medskip
To give more sophisticated consequences related to cluster tilting theory or the theory of non-commutative resolutions, we also introduce a stronger version, which we call {\it strictly hereditary representation type}.
Let $R$ be a commutative ring which we assume to be a Gorenstein local isolated singularity of dimension $d+1$ over a field $k$. Then Auslander-Reiten duality implies that the stable category $\sCM R$ is a $d$-Calabi-Yau triangulated category in the sense that the $d$-th suspension functor $[d]$ is the Serre functor. Such triangulated categories are of substantial importance in numerous areas of mathematics. In representation theory and around, those endowed with cluster tilting objects, called the {cluster categories}, are of particular interest. They play an essential role in categorification of cluster algebras, and also provides a deep connection between singularity theory. We refer to Section \ref{NCCR} for more details which will be used in this paper.
Now Morita theorem for Calabi-Yau triangulated categories \cite{ha4} (see Theorem \ref{jus}) provides the canonical form of $d$-Calabi-Yau triangulated categories arising from hereditary algebras, which leads to the following definition.
\begin{Def}[{Definition \ref{strict}}]\label{ishered}
Let $R$ be a commutative complete Gorenstein local isolated singularity of dimension $d+1\geq3$. We say that $R$ is of {\it strictly hereditary representation type} if there exists a triangle equivalence
\[ \sCM R \simeq \Db(\mod H)/\tau^{-1/(d-1)}[1] \]
for some finite dimensional hereditary algebra $H$ and the naturally defined $(d-1)$-st root of the Auslander--Reiten translation.
\end{Def}
Compared to Definition \ref{ihered}, we require that the autoequivalence $F$ has to be of the form $\tau^{-1/(d-1)}[1]$, which is best possible in the sense that the right-hand-side has a natural $d$-cluster tilting object and is a certain variation of the $d$-cluster category of $H$ \cite{KMV,ha3,haI}; see Definition \ref{strict} and Remark \ref{remark} on the behavior of $\tau^{-1/(d-1)}$ which assures these implications.

We discuss some consequences of a ring being strictly hereditary representation type. We note the uniqueness of enhancements of the singularity category (Theorem \ref{ue}).
Also, taking full advantage of the algebra $H$ being hereditary, we provide a first evidence for the non-commutative Bondal--Orlov conjecture (Theorem \ref{smd}, also Remark \ref{bo}); we show that if $R$ is of strictly hereditary representation type of dimension $d+1\geq3$, then for any $d$-cluster tilting object has the same number of direct summands.

We refer to Section \ref{else} for possible variations of the definitions of hereditary representation types.
\medskip

It is a non-trivial task to construct commutative rings of (strictly) hereditary representation type. To the best of the author's knowledge, there are only two known such examples (see Theorem \ref{exIY}) which are given as certain finite quotient singularities. In the following main computational results of this paper, we present one family of $1$-dimensional semigroup rings of hereditary representation type, and three toric singularities of strictly hereditary representation type.
\begin{Thm}[{Theorem \ref{num}}]\label{exn}
Let $k$ be a field of characteristic $0$. For each $n\geq3$, the ring
\[ k[[x_0,\ldots,x_{n-2}]]/(x_ix_j-x_lx_m\mid i+j=l+m\, (\mod n)) \]
is of hereditary representation type. In fact, its singularity category is equivalent to the $0$-cluster category of the $n$-subspace quiver below. 
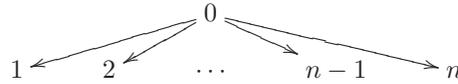
\[  \xymatrix@R=3mm{
	&&0\ar[dll]\ar[dl]\ar[dr]\ar[drr]&&\\
	1&2&\cdots&n-1&n } \]
\end{Thm}

For (the completion of) a graded ring $A$ we denote by $A^{(n)}$ the $n$-th Veronese subring of $A$.
\begin{Thm}[{Section \ref{MORITA}}]\label{Segres}
Let $k$ be an arbitrary field. The following rings are of strictly hereditary representation type.
\begin{enumerate}
	\item Segre product $k[[x,y]]^{(2)}\seg k[[x,y]]^{(2)}$ with $\deg x=\deg y=1$.
	\item Segre product $k[[x,y,z]]\seg k[[u,v]]$ with $\deg x=\deg y=\deg z=1$ and $\deg u=1$, and $\deg v=2$.
	\item Segre product $k[[x,y,z]]\seg k[[x,y,z]]$ with $\deg x=\deg y=\deg z=1$.
\end{enumerate}
In fact, the singularity categories of the above rings are equivalent to the cluster categories arising from the following quivers.
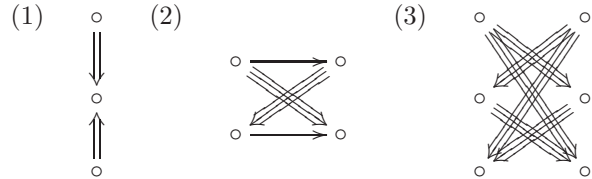
\[ \xymatrix@R=1.2mm@C=4mm{
	{\rm (1)} &\circ\ar@2[dd]& {\rm(2)}&&&& {\rm(3)}& \circ\ar@3[ddrr]\ar@3[ddddrr]&&\circ\ar@3[ddll]\ar@3[ddddll]\\
	&&& \circ\ar[rr]\ar@3[ddrr]&&\circ\ar@3[ddll] &&&&\\
	&\circ&&&&&& \circ\ar@3[ddrr]&&\circ\ar@3[ddll]\\
	&&& \circ\ar[rr]&&\circ &&&&\\
	&\circ\ar@2[uu]&&&&&& \circ&&\circ } \]	
\end{Thm}
The ring in Theorem \ref{exn} is a $1$-dimensional Gorenstein non-domain for each $n$. In fact, it is an example of {\it extended numerical semigroup rings}, which we introduce as a generalization of the classical numerical semigroup rings (Definition \ref{exnum}). An application of tilting theory for rings of dimension $1$ \cite{BIY,haI} shows that it is of hereditary representation type; in fact the triangle equivalence we give (Theorem \ref{num}) shows that its Cohen-Macaulay representation is essentially equivalent to the classical $n$-subspace problem.

The results in Theorem \ref{Segres} are obtained as applications of Morita theorem for Calabi-Yau triangulated categories, as recalled in Theorem \ref{morita}. To apply this theorem we need to construct a cluster tilting object on our rings, which is achieved in a much greater generality as we explain in the following subsection.
\subsection{Non-commutative crepant resolutions and cluster tilting for Segre products}
Let $R$ be a commutative Gorenstein local ring of dimension $d+1\geq2$ which is an isolated singularity. Recall that a {\it $d$-cluster tilting object}, in the sense of Iyama \cite{Iy07a}, is $M\in\CM R$ satisfying
\[
\begin{aligned}
	\add M&=\{ X\in \CM R \mid \Ext^i_R(M,X)=0 \text{ for } 0<i<d\}\\
	&=\{X\in \CM R \mid \Ext^i_R(X,M)=0 \text{ for } 0<i<d \}.
\end{aligned}
\]
It serves as a natural domain for higher dimensional Auslander--Reiten theory \cite{Iy07a,Iy07b}, which is actively studied in representation theory and also plays a crucial role in categorification of cluster algebras \cite{BMRRT}. 
At the same time, cluster tilting objects are special types of {\it non-commutative crepant resolutions} (NCCR for short), in the sense of Van den Bergh \cite{VdB04}; recall that an NCCR of a commutative Gorenstein normal domain $R$ is a reflexive module $M$ whose endomorphism ring $\End_R(M)$ has finite global dimension and is Cohen-Macaulay as an $R$-module. The non-commutative ring $\End_R(M)$ is regarded as a virtual space which is represents a crepant resolution of the singularity $\Spec R$, and in this way the study of non-commutative rings is an important subject in birational geometry.
These cluster tilting objects or NCCRs are characterized by distinguished homological properties of the endomorphism rings $\End_R(M)$, called Calabi--Yau algebras \cite{Gi,Ke11}. 
Such a variety of aspects of cluster tilting objects or NCCRs attracts numerous studies on the subject, from such viewpoints as representation theory, algebraic geometry, commutative algebra, and combinatorics. We refer to \cite{Br,BLV1,BIKR,DFI,FMS,Hara,HIMO,HN,Iy07b,SV17,SV20a,VdB04} and also surveys \cite{Le12,We16,VdB22}, for some of the studies on NCCRs or cluster tilting modules.

Theorem \ref{Segres} is based on a general existence theorem of non-commutative crepant resolutions on Segre products.
Let $k$ be a perfect field, $R^\prime=\bigoplus_{i\geq0}R^\prime_i$ and $R^\pprime=\bigoplus_{i\geq0}R^\pprime_i$ positively graded commutative Gorenstein normal domains of dimension $\geq2$ such that $R_0^\prime$ and $R_0^\pprime$ are finite dimensional over $k$. 
Let $R$ be the Segre product, that is,
\[ R=\bigoplus_{i\geq0}R^\prime_i\otimes_k R^\pprime_i. \]
When $R^\prime$ and $R^\pprime$ have the same negative $a$-invariant, then $R$ is again Gorenstein with still the same $a$-invariant (Corollary \ref{GP}).
Our general main result is the following construction of non-commutative crepant resolutions over the Segre product.
\begin{Thm}[{Theorem \ref{CT}}]\label{iCT}
Suppose that $R^\prime$ and $R^\pprime$ have the same $a$-invariant $-p$. If $R^\prime$ (resp. $R^\pprime$) has an NCCR $M$ (resp. $M^\pprime$) such that $\End_{R^\prime}(M^\prime)$ (resp. $\End_{R^\pprime}(M^\pprime)$) is positively graded, then $\bigoplus_{l=0}^{p-1}M^\prime(l)\seg M^\pprime$ is an NCCR for $R$.
\end{Thm}
Note that existence of an NCCR whose endomorphism ring is positively graded implies that the $a$-invariant is negative (Lemma \ref{a}), thus the direct sum always makes sense. 
We also give the version of Theorem \ref{iCT} with ``cluster tilting'' in the place of NCCR, see Theorem \ref{CT} for details.

Trivially, if $S$ is a regular ring, then $S\in\CM S$ is an NCCR. Applying Theorem \ref{iCT} to this very special case has the following general consequence.
\begin{Cor}[{Corollary \ref{poly}}]
Let $S_i=k[x_{i,0},\ldots,x_{i,d_i}]$, $1\leq i\leq n$ be the polynomial rings with $\deg x_{i,j}=a_{i,j}>0$. Suppose that $\sum_{j=0}^{d_i}a_{i,j}$ is common for all $i$. Then the Segre product $S_1\seg\cdots\seg S_n$ has an NCCR.
\end{Cor}
We mention that Theorem \ref{iCT} is obtained as an operation on the Calabi--Yau algebras in the following sense. Given two graded rings and an integer $p>0$, we introduce their {\it $p$-Segre product} (Definition \ref{ps}) which contains the usual Segre product as an idempotent subring. We prove that the $p$-Segre product of two Calabi-Yau algebras $\End_{R^\prime}(M^\prime)$ and $\End_{R^\pprime}(M^\pprime)$ with the same $a$-invariant $-p$ gives rise to a new Calabi-Yau algebra, which turns out to be the endomorphism ring of $\bigoplus_{l=0}^{p-1}M^\prime(l)\seg M^\pprime$ over the Segre product.

Note that Calabi-Yau algebras have intimate relationship between a class of finite dimensional algebras called higher representation infinite algebras. An independent proof of Theorem \ref{iCT} based on such algebras will appear in our forthcoming paper \cite{segre}. 

\subsection{Non-Gorenstein rings}
Finally, we also discuss non-Gorenstein examples. Due to the lack of a ``model'' of the stable category $\sCM R$ as derived or cluster categories, it is not possible at this point to give an analogous definition for non-Gorenstein cases as in Definition \ref{ihered} and \ref{ishered}. Nevertheless, one can still find some examples where Cohen-Macaulay representations of a non-Gorenstein ring is controlled by a finite dimensional hereditary algebra. Here we restrict our attention to the following ring.
Let
\[ R=k[[x_0,x_1]]\seg k[[y_0,y_1,y_2]] \]
be (the completion of) the Segre product with standard gradings. It is a $4$-dimensional Cohen-Macaulay isolated singularity which is not Gorenstein. 
We construct a $2$-cluster tilting object in $\CM R$, which consequently yields that the Cohen-Macaulay representations of is controlled by the representations of the $3$-Kronecker quiver, hence $R$ can be seen as ``hereditary representation type''. We denote by $\om$ the canonical module, and by $\Om$ the syzygy, that is, the kernel of the projective cover in $\mod R$.
\begin{Thm}\label{CM}
Let $R=k[[x_0,x_1]]\seg k[[y_0,y_1,y_2]]$.
\begin{enumerate}
\item(Theorem \ref{main}) The module $X=R\oplus\om\oplus\Om^2\om\in\CM R$ is $2$-cluster tilting.
\item\label{cor}(Corollary \ref{kQ_3}) There exists an additive equivalence
\[ \xymatrix{ \sHom_R(\Om X,-)\colon \CM R/[\add X]\ar[r]^-\simeq&\mod kQ_3 }\]
for the Kronecker quiver $\xymatrix{Q_3\colon\circ\ar@3[r]&\circ}$ with $3$ arrows.
\end{enumerate}
\end{Thm}
This should be compared with the study by Auslander-Reiten \cite{AR89} (see also \cite{Yo90}), where they discover $3$-dimensional non-Gorenstein rings of finite representation type, in other words, with $1$-cluster tilting module. The above theorem can be seen as a higher dimensional analogue in the sense that we have a $2$-cluster tilting object in a $4$-dimensional ring.

The statement Theorem \ref{CM}(2) of being hereditary representation type in an undefined sense is still strong enough to give a classification result. Recall that a module $M$ over a ring $A$ is {\it rigid} if $\Ext^1_A(M,M)=0$. 
As was accomplished by Iyama--Yoshino \cite{IYo} for some Gorenstein quotient singularities, being hereditary representation type for some special quivers enables us to give a classification result.
It turns out that $X$ given above is the {\it unique} $2$-cluster titling object, and also there are only {\it finitely many} rigid Cohen-Macaulay modules. We say a module is {\it basic} if it has no multiple direct summands.
\begin{Thm}
Let $R=k[[x_0,x_1]]\seg k[[y_0,y_1,y_2]]$ as above.
\begin{enumerate}
\item $R\oplus \om\oplus\Om^2\om$ is the unique basic $2$-cluster tilting object in $\CM R$.
\item A basic Cohen-Macaulay $R$-module is rigid if and only if it is a direct summand of one of the following.
\[ R\oplus\om\oplus\Om^2\om, \quad R\oplus\Om\om,\quad \om\oplus \Hom_R(\Om\om,\om) \]
\end{enumerate}
\end{Thm}

Let us stress that what makes the above classification possible is the theory of representations of quivers, which therefore certifies the significance of pursuing the relationship between representations Cohen-Macaulay rings and finite dimensional algebras, especially Cohen-Macaulay rings of hereditary representation type.

Further examples and generalizations shall be discussed in \cite{segre}.

\subsection*{Acknowledgements}
The author would like to thank Osamu Iyama, Martin Kalck, and Yusuke Nakajima for many valuable discussions.
He is grateful to Yuji Yoshino for bringing gradability to his attention.

\section{Preliminaries on Segre products}
All rings in this section are assumed to be commutative.

We fix a base ring $k$. Let $A=\bigoplus_{i\in\Z}A_i$ and $B=\bigoplus_{i\in\Z}B_i$ be graded algebras over $k$. Then the {\it Segre product} of $A$ and $B$ over $k$ is the graded ring
\[ A\seg_kB=A\seg B:=\bigoplus_{i\in\Z}A_i\otimes B_i. \]
This is a homogeneous coordinate ring of the product $\Proj A\times_{\Spec k}\Proj B$.
Similarly for graded modules $M\in\Mod^\Z\!A$ and $N\in\Mod^\Z\!B$, we define a graded $(A\seg B)$-module
\[ M\seg N:=\bigoplus_{i\in\Z}M_i\otimes N_i. \]
We shall often regard the Segre product as a Veronese subring in the following way. One can naturally view the tensor product $A{\otimes}B=\bigoplus_{i,j\in\Z}A_i\otimes B_j$ as a $\Z^2$-graded ring by $\deg(a\otimes b)=(\deg a,\deg b)$. Then the Segre product is nothing but the Veronese subring corresponding to the diagonal subgroup $\{(i,i)\mid i\in\Z\}$ of $\Z^2$.

In what follows we assume that $k$ is a field. Let $A=\bigoplus_{i\geq0}A_i$ be a positively graded Cohen-Macaulay ring of dimension $d$ such that $A_0$ is a finite dimensional local $k$-algebra, and denote by $\om=\om_A$ the graded canonical module of $A$. We say that $A$ has {\it $a$-invariant $a$} if there exists an isomorphism
\[ \RHom_A(A_0,\om_A)[d]\simeq A_0(-a) \]
in $\rD(\Mod^\Z\!A)$. This is equivalent to saying that the socle of the injective hull of the simple $A$-module (concentrated in degree $0$) lies in degree $a$.

The following result of Goto--Watanabe on Segre products is fundamental, as this computation of the local cohomology groups has many important implications.
Let $k$ be a field, and let $A=\bigoplus_{i\geq0}A_i$ and $B=\bigoplus_{i\geq0}B_i$ be a positively graded commutative Noetherian rings with such that $A_0$ and $B_0$ are finite dimensional local $k$-algebras. We denote by $\m=J_{A_0}\oplus\bigoplus_{i>0}A_i$ and $\n=J_{B_0}\oplus\bigoplus_{i>0}B_i$ the maximal ideals, where $J$ means the Jacobson radical.
\begin{Thm}[{\cite[Theorem 4.1.5]{GW1}}]\label{lc}
Let $M\in\Mod^\Z\!A$ and $N\in\Mod^\Z\!B$. Suppose $\RH_\m^i(M)=0$ for $i=0,1$. Then we have
\[ \RH^p_{\m\seg\n}(M\seg N)=\RH_\m^p(M)\seg N \oplus M\seg\RH^p_\n(N) \oplus \bigoplus_{i=1}^p\RH_\m^i(M)\seg\RH_\n^{p+1-i}(N) \]
for all $p\geq0$.
\end{Thm}
As we shall use later for the case $A$ and $B$ are polynomial rings, this yields ``Cohen-Macaulayness of modules of covariants'' which has been studied from various perspectives e.g.\! \cite{HR,St,VdB93}. For now let us note the following consequence which we will frequently use.
\begin{Cor}[{\cite[Theorem 4.2.3, 4.3.1]{GW1}}]\label{GP}
Let $A=\bigoplus_{i\geq0}A_i$ and $B=\bigoplus_{i\geq0}B_i$ be positively graded Cohen-Macaulay rings such that $A_0$ and $B_0$ are finite dimensional local $k$-algebras, $\dim A=d\geq2$, and $\dim B=e\geq1$. 
\begin{enumerate}
\item Suppose that the $a$-invariants of $A$ and $B$ negative. Then $A\seg B$ is Cohen-Macaulay of dimension $d+e-1$, and $\om_{A\seg B}=\om_A\seg \om_B$.
\item If $A$ and $B$ are Gorenstein with the same negative $a$-invariant, then $A\seg B$ is also Gorenstein with the same $a$-invariant. \end{enumerate}
\end{Cor}

We shall be looking at modules over the Segre product $A\seg B$, among which typical ones are direct summands $A\otimes B$.
Let $A=\bigoplus_{i\geq0}A_i$ and $B=\bigoplus_{i\geq0}B_i$ be positively graded Noetherian domains, and let $S=A\otimes B$, on which we give a $\Z^2$-grading by $\deg(a\otimes b)=(\deg a,\deg b)$. For a $\Z^2$-graded $S$-module $X$ we let $X_i=\bigoplus_{j\in\Z}X_{(i,0)+(j,j)}$. In particular, $S_0=\bigoplus_{j\in\Z}S_{(j,j)}=:R$ is the Segre product $A\seg B$, and $S_i=\bigoplus_{j\in\Z}S_{(i,0)+(j,j)}$ is equal to $A(i)\seg B$ as $\Z$-graded $R$-modules. 

The following result should be standard, and the proof can be found for example in \cite[Lemma 16.7]{Yo90}. Recall that a {\it normal ring} is a commutative ring whose localization at any prime ideal is a normal domain. Any Noetherian normal ring is a direct product of finitely many Noetherian normal domains, and vice versa.
\begin{Lem}\label{cov}
Assume that $S$ is a Noetherian normal ring. Then for reflexive $N\in\mod^{\Z^2}\!S$, the natural map
\[ \xymatrix{ N_{j-i}\ar[r]&\Hom_R(S_i,N_j) } \]
is an isomorphism if any associated prime of $S/SS_i$ has height $\geq2$.
\end{Lem}
\begin{proof}
	See \cite[Lemma 16.7]{Yo90} for $N=S$. For general $N$, we can take an exact sequence $0\to N\to P^0\to P^1$ in $\mod^{\Z^2}\!S$ with $P^0,P^1\in\proj^{\Z^2}\!S$ since $N$ is reflexive. Then use five lemma.
\end{proof}
The above result leads to the following useful observation. Let $k$ be a perfect field, and let $A=\bigoplus_{i\geq0}A_i$ and $B=\bigoplus_{i\geq0}B_i$ be positively graded Noetherian normal domains of dimension $\geq2$ such that $A_0$ and $B_0$ are finite dimensional over $k$. As above let $R=A\seg B$ the Segre product.
\begin{Prop}\label{five}
Let $M_1,M_2\in\mod^\Z\!A$ and $N_1,N_2\in\mod^\Z\!B$ such that $M_2$ and $N_2$ are reflexive. 
Then the canonical map
\[ \xymatrix{ \Hom_{A}(M_1,M_2)\seg\Hom_{B}(N_1,N_2)\ar[r]&\Hom_{R}(M_1\seg N_1,M_2\seg N_2) } \]
is an isomorphism.
\end{Prop}
\begin{proof}
	We know that $S=A\otimes B$ is a Noetherian normal ring since $k$ is perfect (see \cite[Chapter V, Section 1.7, Corollary to Proposition 19]{NB}). Let us verify that the associated primes of $S/S_iS$ have height $\geq2$. Consider first the case $i\geq0$. It is easy to see that $S_i$ generates the ideal $A_{\geq i}\otimes B$, so $S/S_iS=A/A_{\geq i}\otimes B$ has codimension $\dim S-\dim B=\dim A\geq2$. Similarly if $i\leq0$ the ring $S/S_iS$ has codimension $\geq\dim B\geq2$.
	
	If $M_1$ and $N_1$ are finitely generated projective, say $M_1=A(i)$ and $N_1=B(j)$, then the map becomes
	\[ \xymatrix{ M_2(-i)\seg N_2(-j)\ar[r]&\Hom_{R}(A(i)\seg B(j),M_2\seg N_2) }. \]
	This can be seen to be an isomorphism by applying Lemma \ref{cov} to $N=M_2\otimes N_2$ which is reflexive since each factor is. The general assertion is now a routine. Take a presentation of $M_1$ and use five lemma to prove for the case $M_1$ is general and $N_1$ is projective. Do the same argument for $N_1$ to deduce the result.
\end{proof}
\section{Non-commutative resolutions for Segre products}
We prove the first main result of this paper. It is a general construction of non-commutative crepant resolutions (NCCRs) or cluster tilting (CT) modules on Segre products of commutative Gorenstein rings from such a module on each component. For this we work over their Auslander algebras, in other words Calabi-Yau algebras. They have distinguished homological properties which serve well enough to characterize NCCRs or CT modules. In fact, we give a construction of a Calabi-Yau algebra from given such two, and therefore deduce our result. 

\subsection{$p$-Segre products}\label{daiji}
We start with a construction of graded rings, which we call {\it $p$-Segre product} for any positive integer $p$, generalizing the classical Segre product. 
In this subsection we let $k$ be a commutative base ring. Every tensor product will be over $k$. For a $\Z$-graded $k$-algebra $\La$ and a subset $I\subset\Z$ we put $\proj^I\!\La:=\add\{\La(-i)\mid i\in I\}\subset\proj^\Z\!\La$, the full subcategory consisting of projective modules generated in degree $i$ with $i\in I$.

Let
\[ \La^\prime=\bigoplus_{i\in\Z}\La^\prime_i, \qquad \La^\pprime=\bigoplus_{i\in\Z}\La^\pprime_i \]
be graded $k$-algebras. We first view the tensor product $\La^\prime\otimes\La^\pprime$ as a $\Z^2$-graded ring. Fix an integer $p>0$ and consider the category
\[ \P=\add\{ \La^\prime\otimes\La^\pprime(i,j)\mid 0\leq i-j\leq p-1\}\subset\proj^{\Z^2}\!\La^\prime\otimes\La^\pprime. \]
By taking a representative, for example $P=\bigoplus_{l=0}^{p-1}\La^\prime\otimes\La^\pprime(l,0)$, for the action $(1,1)$ on $\P$, we have $\P\simeq\proj^\Z\!\La$ for the graded algebra
\[ \La=\bigoplus_{l\in\Z}\Hom_{\La^\prime\otimes\La^\pprime}^{\Z^2}(P,P(l,l)). \]
\begin{Def}\label{ps}
We call $\La$ above the {\it $p$-Segre product} of $\La^\prime$ and $\La^\pprime$.
\end{Def}
For $p=1$ this is nothing but the usual Segre product. 
The above definition endows the $p$-Segre product $\La$ a structure of a graded ring. It is easy to see that if $\La^\pprime$ is positively graded, then $\La$ is also positively graded. Note however that the grading depends on the choice of the representative $P$, hence it is not canonical. (For example, another choice of a representative $\bigoplus_{l=0}^{p-1}\La^\prime\otimes\La^\pprime(0,-l)$ yields a different grading.) Nevertheless, as an ungraded ring, $\La$ is uniquely determined up to Morita equivalence by the graded rings $\La^\prime$ and $\La^\pprime$, since $\proj\La$ is equivalent to the orbit category $\P/(1,1)$.

	
\bigskip

In what follows we assume that
\[ \La^\prime=\bigoplus_{i\geq0}\La^\prime_i, \qquad \La^\pprime=\bigoplus_{i\geq0}\La^\pprime_i \]
are positively graded, and are flat over $k$. Next fix $S^\prime\in\Mod\La^\prime_0$ (resp. $S^\pprime\in\Mod\La^\pprime_0$), which we view as a graded module over $\La^\prime$ (resp. $\La^\pprime$) concentrated in degree $0$. We assume that they have projective resolutions
\[ \xymatrix@R=2mm{
	P^\prime\colon&0\ar[r]&P^\prime_{m^\prime}\ar[r]&P^\prime_{{m^\prime}-1}\ar[r]&\cdots\ar[r]&P^\prime_1\ar[r]&P^\prime_0\ar[r]&0 \\
	P^\pprime\colon&0\ar[r]&P^\pprime_{m^\pprime}\ar[r]&P^\pprime_{{m^\pprime}-1}\ar[r]&\cdots\ar[r]&P^\pprime_1\ar[r]&P^\pprime_0\ar[r]&0 }\]
in $\mod^\Z\!\La^\prime$ and $\mod^\Z\!\La^\pprime$ respectively, satisfying the following.
\begin{itemize}
	\item The first term $P_0^\prime$ (resp. $P_0^\pprime$) is generated in degree $0$.
	\item The middle terms are generated in degrees $[0,p]$.
	\item The last term $P^\prime_{m^\prime}$ (resp. $P^\pprime_{m^\pprime}$) is generated in degree $p$.
\end{itemize}

Now from these two complexes we construct a projective resolution of the $\La$-module $S^\prime(l)\otimes S^\pprime$ for each $0\leq l\leq p-1$.
First truncate the complex $P^\prime(l)$ (which is quasi-isomorphic to $S^\prime(l)$) along the stable $t$-structure $\Kb(\proj^\Z\!\La^\prime)=\Kb(\proj^{\leq0}\!\La^\prime)\perp\Kb(\proj^{>0}\!\La^\prime)$ to realize it as a mapping cone
\[ \xymatrix{ X^\prime\ar[r]& Y^\prime\ar[r]& S^\prime(l)\ar[r]&X^\prime[1] } \]
in $\per^\Z\!\La^\prime$ with $X^\prime\in\Kb(\proj^{[1,p-l]}\!\La^\prime)$ and $Y^\prime\in\Kb(\proj^{[-l,0]}\!\La^\prime)$.
Similarly, truncate the second complex $P^\pprime$ along $\Kb(\proj^\Z\!\La^\pprime)=\Kb(\proj^{<p-l}\!\La^\pprime)\perp\Kb(\proj^{\geq p-l}\!\La^\pprime)$ to obtain a triangle
\[ \xymatrix{ X^\pprime\ar[r]& Y^\pprime\ar[r]& S^\pprime\ar[r]&X^\pprime[1] } \]
in $\per^\Z\!\La^\pprime$ with $X^\pprime\in\Kb(\proj^{[p-l,p]}\!\La^\pprime)$ and $Y^\pprime\in\Kb(\proj^{[0,p-l-1]}\!\La^\pprime)$.
\begin{Prop}\label{resolution}
	The restriction of the sequence
	\[ \xymatrix{ X^\prime\otimes X^\pprime\ar[r]&Y^\prime\otimes Y^\pprime\ar[r]& S^\prime\otimes S^\pprime(l,0)\ar[r]& X^\prime\otimes X^\pprime[1] } \]
	in $\rD(\Mod^{\Z^2}\!\La^\prime\otimes\La^\pprime)$ to $\P$ yields a triangle in $\rD(\Mod^\Z\!\La)$. Therefore, the mapping cone of the first map gives a projective resolution of $S^\prime(l)\otimes S^\pprime$ in $\Mod^\Z\!\La$.
\end{Prop}
\begin{proof}
	Extend the first map to the triangle
	\[ \xymatrix{ X^\prime\otimes X^\pprime\ar[r]&Y^\prime\otimes Y^\pprime\ar[r]& C\ar[r]& X^\prime\otimes X^\pprime[1] } \]
	in $\rD(\Mod^{\Z^2}\!\La^\prime\otimes\La^\pprime)$ and we show that restricting $C$ to $\P$, in other words, taking the degree $\{(i,j)\in\Z^2\mid 0\leq j-i\leq p-1\}$ part (notice the difference of the signs), gives $S^\prime(l)\otimes S^\pprime$, hence its resolution.
	
	First, since $X^\prime\in\Kb(\proj^{[1,p-l]}\!\La^\prime)$ and $X^\pprime\in\Kb(\proj^{[p-l,p]}\!\La^\pprime)$, we have that each term of $X^\prime\otimes X^\pprime$ lies in $\add\{\La^\prime\otimes\La^\pprime(-i,-j)\mid 1\leq i\leq p-l,\, p-l\leq j\leq p\}$. It follows that restricting them to $\P$ gives graded projective $\La$-modules (precisely, in $\add\{\La(-j)\mid p-l\leq j\leq p\}$). Similarly, by the locations of $Y^\prime$ and $Y^\pprime$, we find that each term of $Y^\prime\otimes Y^\pprime$ restricts to $\add\{\La(-j)\mid 0\leq j\leq p-l-1\}$. 
	
	Let us next compute the cohomology of $C$. The octahedral axiom applied to the composite $X^\prime\otimes X^\pprime\to Y^\prime\otimes X^\pprime\to Y^\prime\otimes Y^\pprime$ yields a triangle
	\[ \xymatrix{ S^\prime\otimes X^\pprime(l,0)\ar[r]& C\ar[r]& Y^\prime\otimes S^\pprime\ar[r]&S^\prime\otimes X^\pprime(l,0)[1] } \]
	in $\rD(\Mod^{\Z^2}\!\La^\prime\otimes\La^\pprime)$.
	For each $n\in\Z$ we have $H^n(S^\prime\otimes X^\pprime(l,0))=H^0S^\prime(l)\otimes H^nX^\pprime$, which is concentrated in degree $\{(-l,j)\mid j\geq p-l\}$ by the location of $X^\pprime$, hence the cohomology is $0$ when restricted to $\P$. Similarly we have $H^n(Y^\prime\otimes S^\pprime)=H^nY^\prime\otimes H^0S^\pprime$. If $n\neq0$ there is an isomorphism $H^nX^\prime\xsimeq H^nY^\prime$ so that this is concentrated in degree $>0$. Therefore $H^n(Y^\prime\otimes S^\pprime)$ is concentrated in degree $\{(i,0)\mid i>0\}$ when $n\neq0$, thus restricts to $0$ on $\P$.
	It remains to consider $H^0(Y^\prime\otimes S^\pprime)=H^0Y^\prime\otimes S^\pprime$. Truncating the exact sequence $H^0X^\prime\to H^0Y^\prime\to S^\prime(l)\to0$ at degree $\leq0$, we get an isomorphism $H^0Y^\prime_{\leq0}\xsimeq S^\prime(l)$ by the location of $X^\prime$, and since each term of $Y^\prime$ is generated in degree $\geq-l$, we have $H^0Y^\prime_{\leq0}=H^0Y^\prime_{[-p+1,0]}$. Noting that $S^\pprime$ is concentrated in degree $0$, it follows that the restriction of $H^0Y^\prime\otimes S^\pprime$ to $\P$ is isomorphic to $H^0Y^\prime_{[-p+1,0]}\otimes S^\pprime=S^\prime(l)\otimes S^\pprime$.
\end{proof}

\subsection{Noetherian algebras of finite global dimension}
As a next step we specialize the above construction to the following setup of module-finite algebras.
\begin{Setup}
Let $k$ be a field and let $R=\bigoplus_{i\geq0}R_i$ be a positively graded commutative Noetherian $k$-algebra such that $R_0$ is finite dimensional.
\begin{enumerate}
	\item $\La^\prime$ and $\La^\pprime$ are positively graded module-finite graded $R$-algebras of finite global dimension $n^\prime$ and $n^\pprime$, respectively. Moreover either $\La^\prime/J^\prime$ or $\La^\pprime/J^\pprime$ is separable over $k$, where $J^\prime$ and $J^\pprime$ are the Jacobson radicals.
	\item $S^\prime=\La^\prime/J^\prime$ and $S^\pprime=\La^\pprime/J^\pprime$.
	\item $m^\prime\geq n^\prime$, $m^\pprime\geq n^\pprime$, and $p>0$ are integers satisfying the two conditions below.
	\begin{itemize}
	\item $\Ext^i_{\La^\prime}(S^\prime,\La^\prime)$ is concentrated in degrees $[-p,0]$ for $0<i<m^\prime$, and in degree $-p$ for $i=m^\prime$.
	\item $\Ext^i_{\La^\pprime}(S^\pprime,\La^\pprime)$ is concentrated in degrees $[-p,0]$ for $0<i<m^\pprime$, and in degree $-p$ for $i=m^\pprime$.
	\end{itemize}
\end{enumerate}
\end{Setup}
For example one can take $m^\prime>n^\prime$ and $m^\pprime>n^\pprime$, so that the third condition is satisfied for any sufficiently large $p$.
Now let
\[ \La=\Oplus_{l\geq0}\Hom_{\La^\prime\otimes\La^\pprime}^{\Z^2}(P,P(l,l)) \]
for $P=\bigoplus_{l=0}^{p-1}\La^\prime\otimes\La^\pprime(l,0)$, the $p$-Segre product.
\begin{Prop}\label{gl.dim}
The $p$-Segre product $\La$ has global dimension $\leq m^\prime+m^\pprime-1$. If $m^\prime=n^\prime$ and $m^\pprime=n^\pprime$, the global dimension is precisely $n^\prime+n^\pprime-1$. Moreover, the projective resolution of the semisimple $\La$-modules $S^\prime(l)\otimes S^\pprime$ are computed as in Proposition \ref{resolution}.
\end{Prop}
\begin{proof}
	Let $P^\prime\to S^\prime$ and $P^\prime\to S^\pprime$ be the minimal projective resolutions in $\mod^\Z\!\La^\prime$ and $\mod^\Z\!\La^\pprime$, respectively. Then by our assumption on the location of the $\Ext$ groups, these resolutions satisfy the conditions on the projective resolution from the previous subsection. Then we get a projective resolution of $S^\prime(l)\otimes S^\pprime\in\mod^\Z\!\La$ in Proposition \ref{resolution}. Since these are all the simple $\La$-modules (up to degree shifts) we see that $\La/J_\La$ has projective dimension $\leq m^\prime+m^\pprime-1$. It follows from \cite[Theorem A.1]{GI} that $\La$ has global dimension $\leq m^\prime+m^\pprime-1$.
\end{proof}

\subsection{Non-commutative resolutions for Segre products}\label{NCCR}
To state one of the main results of this paper, we recall some important notions related to non-commutative resolutions. We denote by $\refl A$ the category of finitely generated reflexive modules over a commutative ring $A$.
\begin{Def}\label{NCR}
Let $R$ be a Gorenstein normal domain and $M\in\refl R$.
\begin{enumerate}
\item\cite{VdB04} $M$ is a {\it non-commutative resolution}, or an NCR, if $\End_R(M)$ has finite global dimension.
\item\cite{VdB04} $M$ is a {\it non-commutative crepant resolution}, or an NCCR, if it is an NCR and $\End_R(M)\in\CM R$.
\item\cite{Iy07a,IW14} $M$ is a {\it CT module} if $M\in\CM R$ and satisfies the following.
\[ \begin{aligned}
	\add M&=\{ X\in \CM R \mid \Hom_R(M,X)\in\CM R\}\\
	&=\{X\in \CM R \mid \Hom_R(X,M)\in\CM R\}
	\end{aligned}
\]
\end{enumerate}
\end{Def}
If $M\in\refl R$ is an NCCR, then the global dimension of $\End_R(M)$ must be $\dim R$. Also by \cite[Corollary 5.9]{IW14} a CT module is precisely an NCCR which contains $R$ as a direct summand. (Note that we do not need $\dim R=3$ for this.) Moreover, when $R$ has only an isolated singularity, CT module is nothing but a $(\dim R-1)$-cluster tilting object in $\CM R$, in the sense of Definition \ref{ct} below. These facts can be stated as ``Auslander correspondence'' below.
\begin{Thm}[\cite{Iy07b,IW14}]
Let $R$ be a Gorenstein local normal domain and $M\in\refl R$. Then the following are equivalent.
\begin{enumerate}
	\renewcommand\labelenumi{(\roman{enumi})}
	\renewcommand\theenumi{(\roman{enumi})}
	\item $M\in\CM R$ and is a CT module.
	\item $R\in\add M$, $\gd\End_R(M)=\dim R$, and $\End_R(M)\in\CM R$.
\end{enumerate}
\end{Thm}


We remark the following example which is trivial but will be useful.
\begin{Ex}
Let $R$ be a regular ring. Then $R\in\CM R$ is a CT module.
\end{Ex}


We need to compare the $a$-invariants of $R$ and its Auslander algebra $\Ga$. Assume that $\dim R=d$, and recall that $\Ga$ has {\it $a$-invariant $a$} if $\RHom_\Ga(S,\Ga)[d]\in\mod^\Z\!\Ga^\op$, and is a simple $\Ga^\op$-module concentrated in degree $a$ for each simple $\Ga$-module $S$.
\begin{Lem}\label{a}
Let $R$ be a positively graded Noetherian normal domain, $M\in\refl R$, and $\Ga=\End_R(M)$.
\begin{enumerate}
\item We have $\Ga\simeq\Hom_R(\Ga,R)$ as graded bimodules.
\item If $M\in\refl R$ is an NCCR, then $R$ and $\Ga$ has the same $a$-invariant, which is negative if $\dim R>0$ and $\Ga$ is positively graded.
\end{enumerate}
\end{Lem}
\begin{proof}
	(1)  This can be seen from that the isomorphism $\Hom_R(\Ga,R)\simeq\Ga$ is canonical, see \cite[Lemma 2.9]{IW14}.
	
	(2)  Note that by (1) we have $\RHom_R(\Ga,R)\simeq\Ga$ in $\rD(\Mod^\Z\!\Ga^e)$ since $\Ga\in\CM R$. Then there is an isomorphism $\RHom_\Ga(-,\Ga)=\RHom_\Ga(-,\RHom_R(\Ga,R))=\RHom_R(-,R)$ on $\rD(\Mod^\Z\!\Ga)$. Since any simple $\Ga$-module is a finite length $R$-module concentrated in degree $0$, the assertion follows. It is easy to see that the $a$-invariant is negative using that $\Ga$ is positively graded and has finite global dimension. 
\end{proof}
We are now ready to state and prove the result. Let $k$ be a perfect field, {let $R^\prime=\bigoplus_{i\geq0}R^\prime_i$ and $R^\pprime=\bigoplus_{i\geq0}R^\pprime_i$ be positively graded} commutative Gorenstein normal domains of dimension $\geq2$ such that $R^\prime_0$ and $R^\pprime_0$ are finite dimensional local $k$-algebras.
\begin{Thm}\label{CT}
Suppose that $R^\prime$ and $R^\pprime$ have the same $a$-invariant $-p<0$, and let $R=R^\prime\seg R^\pprime$ the Segre product.
Let $M^\prime\in\refl R^\prime$ and $M^\pprime\in\refl R^\pprime$ be graded modules such that $\End_{R^\prime}(M^\prime)$ and $\End_{R^\pprime}(M^\pprime)$ are positively graded, and put $M^q=\bigoplus_{l=0}^{q-1}M^\prime(l)\seg M^\pprime$ for each $q>0$.  
\begin{enumerate}
	\item If $M^\prime$ and $M^\pprime$ are NCRs, then $M^q$ is an NCR for any sufficiently large $q$.
	\item If $M^\prime$ and $M^\pprime$ are NCCRs, then $M^p$ is an NCCR for $R$.
	\item If $M^\prime$ and $M^\pprime$ are CT modules, then $M^p$ is a CT module for $R$.
\end{enumerate}
\end{Thm}
\begin{proof}
	Let $\dim R^\prime=d^\prime+1$, $\dim R^\pprime=d^\pprime+1$, and $\Ga^\prime=\End_{R^\prime}(M^\prime)$, $\Ga^\pprime=\End_{R^\pprime}(M^\pprime)$. We know by Corollary \ref{GP}(2) that $R:=R^\prime\seg R^\pprime$ is Gorenstein of dimension $d+1:=d^\prime+d^\pprime+1$ with $a$-invariant $-p$. Also we know that $R^\prime\otimes R^\pprime$ is normal since each factor is. 
	
	
	By Proposition \ref{five} we find that the Segre product of reflexive modules are reflexive, hence $M^q\in\refl R$.
	Notice that the endomorphism ring of $M^q$ is the $q$-Segre product of $\Ga^\prime$ and $\Ga^\pprime$ for any $q>0$.
	Indeed, we have $\Hom_R(M^\prime(i)\seg M^\pprime,M^\prime(j)\seg M^\pprime)=\Hom_{R^\prime}(M^\prime(i),M^\prime(j))\seg\Hom_{R^\pprime}(M^\pprime,M^\pprime)=\Ga^\prime(j-i)\seg\Ga^\pprime$ by Proposition \ref{five}, while the components of the $q$-Segre product of $\Ga^\prime$ and $\Ga^\pprime$ is $\Hom^\Z_{\Ga^\prime\otimes\Ga^\pprime}(\Ga^\prime(i)\otimes\Ga^\pprime,\Ga^\prime(j)\otimes\Ga^\pprime)=(\Hom_{\Ga^\prime}(\Ga^\prime(i),\Ga^\prime(j))\otimes\Hom_{\Ga^\pprime}(\Ga^\pprime,\Ga^\pprime))_0=(\Ga^\prime(j-i)\otimes\Ga^\pprime)_0=\Ga^\prime(j-i)\seg\Ga^\pprime$. 
	This assertion immediately yields (1) by Proposition \ref{gl.dim}.
	
	In what follows assume $M^\prime$ and $M^\pprime$ are NCCRs. We have to prove that $\End_R(M^p)$ is Cohen-Macaulay as an $R$-module. By the above computation this is to show that $\Ga^\prime(l)\seg\Ga^\pprime\in\CM R$ for $-p<l<p$. By Theorem \ref{lc} we know that $\RH_{\m^\prime\seg\m^\pprime}^{q}(\Ga^\prime(l)\seg \Ga^\pprime)$ is isomorphic to the direct sum of the following three modules:
	\[ \RH_{\m^\prime}^{q}(\Ga^\prime(l))\seg \Ga^\pprime, \qquad \Ga^\prime(l)\seg\RH_{\m^\pprime}^{q}(\Ga^\pprime),\qquad \bigoplus_{i=1}^{q}\RH_{\m^\prime}^i(\Ga^\prime(l))\seg\RH_{\m^\pprime}^{q+1-i}(\Ga^\pprime). \]
	Since $\Ga^\prime\in\CM R^\prime$ and $\Ga^\pprime\in\CM R^\pprime$, the third one vanishes for $q<d+1$. For the first one, we only have to look at $\RH^{d^\prime+1}_{\m^\prime}(\Ga^\prime)$. By local duality and Lemma \ref{a} we have $\RH_{\m^\prime}^{d^\prime+1}(\Ga^\prime)=(D\Hom_{R^\prime}(\Ga^\prime,R^\prime))(p)=D\Ga^\prime(p)$, so that this is concentrated in degree $\leq-p$. It follows that $\RH_{\m^\prime}^{q}(\Ga^\prime(l))\seg \Ga^\pprime=0$ for all $l>-p$. Similarly, the module $\Ga^\prime(l)\seg\RH_{\m^\pprime}^{d^\pprime+1}(\Ga^\pprime)=0$ whenever $l<p$. We therefore conclude that $\End_R(M^p)\in\CM R$, proving (2).
	
	Finally, if $M^\prime$ and $M^\pprime$ are CT modules, then they contain projective summands, hence so does $M^p$. Therefore $M^p$ is a CT module over $R$. This proves (3).	
%
%
%
%
\end{proof}



Even if two rings $R^\prime$ and $R^\pprime$ has different $a$-invariants, one can multiply the gradings to adjust the $a$-invariants to the same value. This observation leads to the following consequence.
\begin{Cor}
Let $R^\prime=\bigoplus_{i\geq0}R^\prime_i$ has $a$-invariant $-p^\prime$ and $R^\pprime=\bigoplus_{i\geq0}R_i^\pprime$ has $a$-invariant $-p^\pprime$. Let $p$ be the least common multiple of $p^\prime$ and $p^\pprime$ with $p/p^\prime=q^\prime$ and $p/p^\pprime=q^\pprime$. Then $(R^\prime)^{(q^\pprime)}\seg (R^\pprime)^{(q^\prime)}$ has an NCCR (resp. a CT module) whenever $R^\prime$ and $R^\pprime$ do.
\end{Cor}
\begin{proof}
For a graded ring $A$ and positive integer $n$ we denote by ${}^n\!A$ the graded ring with multiplied grading, that is, $({}^n\!A)_i=A_{i/n}$ if $n\mathrel{|}i$, and $0$ otherwise. Clearly this construction multiplies the $a$-invariant.
We can apply Theorem \ref{CT} to ${}^{q^\prime}\!R^\prime$ and ${}^{q^\pprime}\!R^{\pprime}$, both of which have $a$-invariant $-p$. Their Segre product is $\bigoplus_{i\geq0}({}^{q^\prime}\!R^\prime)_i\otimes({}^{q^\pprime}\!R^\pprime)_i$, whose $i$-th component do not vanish only if $i$ is a multiple of both $q^\prime$ and $q^\pprime$. Now since they are relatively prime, we can write the Segre product as $\bigoplus_{j\geq0}R^\prime_{q^\pprime j}\otimes R^\pprime_{q^\prime j}$, which is nothing but $(R^\prime)^{(q^\pprime)}\seg (R^\pprime)^{(q^\prime)}$.
\end{proof}

Recall that for a regular local ring $S$, the free module $S\in\CM S$ is trivially a CT module. Consequently we obtain a generalization of \cite{HN} from the standard grading.
\begin{Cor}\label{poly}
Let $k$ be an arbitrary field, and let $S_i=k[x_{i,0},\ldots,x_{i,d_i}]$ be the polynomial rings with $\deg x_{i,j}=a_{i,j}>0$. Suppose that $\sum_{j=0}^{d_i}a_{i,j}$ is common for all $1\leq i\leq n$. Then the Segre product $S_1\seg\cdots\seg S_n$ has a CT module consisting of rank $1$ modules.
\end{Cor}
\begin{proof}
	Note that we do not need that $k$ is perfect since it is only used in the proof of Theorem \ref{CT} to ensure that $R^\prime\otimes R^\pprime$ is normal. The assertion on the rank is clear from the construction.
\end{proof}

Let us note that even if both $R^\prime$ and $R^\pprime$ has only an isolated singularity (even regular) the Segre product might not be an isolated singularity. (We therefore need the notion of CT modules, not just cluster titling objects.)
\begin{Ex}
Let $R^\prime=k[x,y]$, $R^\pprime=k[u,v]$ with $\deg x=\deg u=1$ and $\deg y=\deg v=2$. These have $a$-invariant $-3$, and $R^\prime\in\CM R^\prime$ and $R^\pprime\in\CM R^\pprime$ are trivially CT modules. Therefore $R:=R^\prime\seg R^\pprime$ (which is isomorphic to $k[xu,x^2v,yu^2,yv]\simeq k[a,b,c,d]/(a^2d-bc)$, a compound $A_2$-singularity) has a CT module $R\oplus M_1\oplus M_2$ with $M_i=R^\prime(i)\seg R^\pprime$, whose Auslander algebra is the $3$-Segre product of $R^\prime$ and $R^\pprime$. It is presented by the following quiver with commutativity relations.
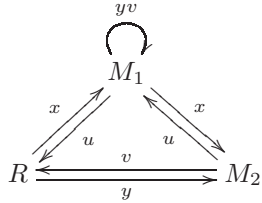
\[ \xymatrix{
	&M_1\ar@<2pt>[dr]^-x\ar@<2pt>[dl]^-u\ar@(ul,ur)[]^-{yv}&\\
	R\ar@<2pt>[ur]^-x\ar@<-2pt>[rr]_-y&&M_2\ar@<2pt>[ul]^-u\ar@<-2pt>[ll]_-v} \]
\end{Ex}

We refer to Section \ref{MORITA} for more examples.


\section{Gorenstein rings of hereditary representation type}
We introduce the class of commutative rings which we shall investigate. Our definitions are motivated by tilting theory and cluster tilting theory, an interaction of representation theories of commutative rings and finite dimensional algebras.
\subsection{Backgrounds on cluster categories and singularity categories}

Let us start with the notion of cluster tilting objects, which plays an important role both in commutative algebra and finite dimensional algebras.
\begin{Def}[{\cite{Iy07a}}]\label{ct}
Let $d\geq1$ an integer and let $\E$ be an exact category with enough projectives and injectives, or a triangulated category. A functorially finite subcategory $\M\subset\E$ is called {\it $d$-cluster tilting} if the following equalities hold.
\[
\begin{aligned}
	\M&=\{ X\in\E \mid \Ext^i_\E(M,X)=0 \text{ for all } M\in\M \text{ and } 0<i<d\}\\
	&=\{ X\in\E \mid \Ext^i_\E(X,M)=0 \text{ for all } M\in\M \text{ and } 0<i<d\}
\end{aligned}
\]
\end{Def}
As we have seen before, we have many examples of commutative rings whose $\CM R$ has a cluster tilting module. On the other hand, the following construction of {\it cluster categories} \cite{BMRRT,Ke05,Am09,Guo,Ke11} and its generalization \cite{KMV,ha3,haI} gives a counterpart for finite dimensional algebras, which consequently serves as a model for the stable categories of Cohen-Macaulay modules \cite{Iy18,haI}.
\begin{Def}
	Let $H$ be a finite dimensional hereditary algebra, $1\neq d\in\Z$ and $a\neq0$. For a bimodule $V\in\per H^e$ such that $V^{\lotimes_Ha}\simeq\RHom_H(DH,H)[d]$ in $\rD(H^e)$, the {\it $a$-folded $d$-cluster category} is the orbit category
	\[ \rC_d^{(1/a)}(H):=\Db(\mod H)/-\lotimes_HV. \]
\end{Def}
Note that we shall use the following special form of $V$ (see Definition \ref{strict}): $U$ is a bimodule such that $U^{\lotimes_H(d-1)}=\RHom_H(DH,H)[1]$ and $V=U[1]$, so that $V^{\lotimes_H(d-1)}\simeq\RHom_H(DH,H)[d]$.

One of the strong implications of derived category of hereditary categories is the following theorem due to Keller \cite{Ke05}, which in particular shows that the above orbit category has a natural structure of a triangulated category (when $d\neq1$).
For additive subcategories $\U, \V$ of an additive category, we denote by $\U\vee\V$ the smallest additive subcategory containing $\U$ and $\V$.
\begin{Thm}[{\cite{Ke05}}]\label{orbit}
	Let $H$ be a finite dimensional hereditary algebra, and let $T$ be a complex of $(H,H)$-bimodules such that $F=-\lotimes_HT$ gives an autoequivalence of $\Db(\mod H)$ satisfying the following.
	\begin{enumerate}
		\renewcommand{\labelenumi}{(\roman{enumi})}
		\renewcommand{\theenumi}{\roman{enumi}}
		\item For each $X\in\mod H$, there are at most finitely many $i\in\Z$ such that $F^iX\in\mod H$.
		\item There exists an integer $n$ such that the $F$-orbit of any indecomposable object of $\Db(\mod H)$ intersects with the subcategory $\mod H\vee(\mod H)[1]\vee\cdots\vee(\mod H)[n]$.
	\end{enumerate}
	Then the orbit category $\Db(\mod H)/F$ has a natural structure of a triangulated category.
\end{Thm}

In what follows we will mainly consider commutative rings in the complete local setting. Representation theoretically it is no harm to do so, as is justified by the following results.
\begin{Prop}[\cite{Or09b}]
Let $R$ be a commutative Gorenstein ring such that the singular locus consists of finitely many maximal ideals $\{\m_1,\ldots,\m_n\}$. We put $R_i$ the $\m_i$-adic completion of $R$. Then $\sCM R\simeq\prod_{i=1}^n\sCM R_i$ up to direct summands.
\end{Prop}
\begin{Prop}[{\cite[Appendix A]{KMV}}]\label{sg's}
Let $R=\bigoplus_{i\geq0}R_i$ be a positively graded commutative Gorenstein ring such that $R_0$ is finite dimensional. Suppose that the singular locus of $R$ is defined by $\m=\bigoplus_{i>0}R_i$ and let $\widehat{R}=\prod_{i\geq0}R_i$ the completion at $\m$. Then the canonical functors
\[ \sCM R\to \sCM R_\m\to \sCM \widehat{R} \]
are equivalences up to summands.
\end{Prop}

\subsection{Definitions and first examples}
Now we are in the position to give a definition of Gorenstein rings of hereditary representation type. Depending on the setting of graded or ungraded/complete, we give the following definition as suggested by tilting theory. 
\begin{Def}\label{hered}
	Let $R$ be a commutative Gorenstein ring over a field $k$.
	\begin{enumerate}
		\item Suppose that $R=\bigoplus_{i\geq0}R_i$ is positively graded such that $R_0$ is finite dimensional local $k$-algebra, and with graded isolated singularity. We say that $R$ is {\it graded hereditary representation type} if there exists a triangle equivalence
		\[ \sCM^\Z\!R\simeq\Db(\mod H) \]
		for some finite dimensional hereditary algebra $H$.
		\item Suppose that $R$ is complete local isolated singularity. We say $R$ is of {\it hereditary representation type} if there exists a triangle equivalence
		\[ \sCM R\simeq \Db(\mod H)/F \]
		for some finite dimensional hereditary algebra $H$ and a triangle autoequivalence $F$ of $\Db(\mod H)$.
	\end{enumerate}
\end{Def}

We shall see that being graded hereditary representation type implies that the completion is of hereditary representation type (see Theorem \ref{gradable} below). 
We also give a stronger version, in which we require that the automorphism $F$ should be the best choice. To justify this we recall the following structure theorem. Recall that a finite dimensional algebra $H$ is {\it $1$-representation infinite} if every connected component of $H$ is a representation infinite hereditary algebra. Also we denote by $J_A$ the Jacobson radical of a ring $A$.
\begin{Thm}[{\cite{ha4}}]\label{jus}
Let $d\geq2$ and $\T$ an algebraic $d$-Calabi-Yau triangulated category with a $d$-cluster tilting object $T$. Suppose that $H=\End_\T(\bigoplus_{i=0}^{d-2}T[-i])$ is $1$-representation infinite and that $H/J_H$ is separable over $k$.
Then there exists a triangle equivalence
\[ \T\simeq\Db(\mod H)/\tau^{-1/(d-1)}[1] \]
for a naturally defined $(d-1)$-st root $\tau^{1/(d-1)}$ of the Auslander-Reiten translation.
\end{Thm}
We refer to \cite[Theorem 4.14]{ha4} for the definition of this root of $\tau$, and to Remark \ref{remark} below for its properties.

By this structure theorem, the right-hand-side is the canonical form of $d$-Calabi-Yau categories arising from hereditary algebras, which leads to the following definition.
\begin{Def}\label{strict}
Let $R$ be a complete Gorenstein local isolated singularity of dimension $d+1\geq3$. We say that $R$ is of {\it strictly hereditary representation type} if there exists a triangle equivalence
\[ \sCM R\simeq\Db(\mod H)/\tau^{-1/(d-1)}[1]=:\rC_d^{(1/(d-1))}(H) \]
for some finite dimensional hereditary algebra $H$, where $\tau^{-1/(d-1)}$ is an autoequivalence of $\Db(\mod H)$ given by a bimodule complex $U\in\per H^e$ satisfying $U^{\lotimes_H(d-1)}\simeq\RHom_H(DH,H)[1]$ in $\rD(H^e)$.
\end{Def}
\begin{Rem}\label{remark}
We know by \cite[Corollary 6.3]{ha4}, replacing $H$ by a derived equivalent hereditary algebra if necessary, that there exists a projective module $P$ satisfying the following.
\begin{itemize}
\item $\add H=\add(\bigoplus_{i=0}^{d-2}P\lotimes_HU^{\lotimes_H i})$,
\item $\Hom_H(P\lotimes_HU^{\lotimes_H i},P)=0$ for $1\leq i\leq d-1$.
\end{itemize}
In this case, the projective module $P$ gives a $d$-cluster tilting object in $\T:=\Db(\mod H)/\tau^{-1/(d-1)}[1]$ such that $\End_\T(\bigoplus_{i=0}^{d-2}P[-i])=H$ (\cite[Corollary 6.4]{ha4}).
\end{Rem}
While hereditary type is a direct ungraded analogue of the graded hereditary type, the notion of strict hereditary type is more closely related to cluster tilting theory and the notion of ``strict tilting'' which we shall introduce and study in a forthcoming work \cite{segre}.

We refer to Section \ref{else} below for possible variations of the definition of hereditary representation type.

\medskip
Recall that a Cohen-Macaulay ring is of {\it finite representation type} if there is an additive generator in $\CM R$. Similarly a graded Cohen-Macaulay ring is called {\it graded finite representation type} if there is an object $M\in\CM^\Z\!R$ such that $\CM^\Z\!R=\add\{ M(i)\mid i\in\Z\}$.

Let us first note that ``finite representation type implies hereditary representation type'', which is an interpretation of structure theorems of finite triangulated categories \cite{Am07,ha} based on the combinatorics of translations quivers \cite{Rie,HPR,XZ}.
Even though some parts of the following results can be seen from the classification in the setting of commutative rings, we state them since these structure theorems hold in much greater generality. 
\begin{Prop}\label{finite}
	Let $R$ be a commutative Gorenstein ring over an algebraically closed field $k$.
	\begin{enumerate}
		\item {\cite[Theorem 7.2]{Am07}} Suppose that $R$ is of finite representation type. Then there exists a triangle equivalence $\sCM R\simeq \Db(\mod kQ)/F$ for a Dynkin quiver $Q$ and a triangle autoequivalence $F$, provided $\sCM R$ is standard. In particular, $R$ is of hereditary representation type.
		\item {\cite[Theorem 7.3]{ha}} Suppose that $R=\bigoplus_{i\geq0}R_i$ is positively graded with $R_0=k$, and that it is of graded finite representation type. Then there exists a triangle equivalence $\sCM^\Z\!R\simeq\Db(\mod kQ)$ for a Dynkin quiver $Q$. In particular, $R$ is of graded hereditary representation type.
	\end{enumerate}
\end{Prop}

Now let us turn to non-trivial examples. To the best of the author's knowledge, the only known examples of commutative Gorenstein rings of hereditary representation type beyond finite representation types are the following, results by Keller--Reiten \cite{KRac} and Keller--Murfet--Van den Bergh \cite{KMV}, based on the examples by Iyama--Yoshino \cite{IYo}, which are our motivation for pursuing commutative rings of hereditary types.
\begin{Thm}\label{exIY}
We denote by $\xymatrix{Q_n\colon\circ\ar[r]^-n&\circ}$ the Kronecker quiver with $n$ arrows between the vertices.
\begin{enumerate}
	\item\cite{IYo,KRac} Let $R=k[[x,y,z]]^{(3)}$. Then there exists a triangle equivalence $\sCM R\simeq\rC_2(kQ_3)$.
	\item\cite{IYo,KMV} Let $R=k[[x,y,z,w]]^{(2)}$. Then there exists a triangle equivalence $\sCM R\simeq\rC_3^{(1/2)}(kQ_6)$.
\end{enumerate}
In fact, these are of strictly hereditary representation type.
\end{Thm}

To discuss some consequences of being (graded) hereditary representation type in the following subsection, here we state a minor remark.
\begin{Lem}\label{a=0}
	Let $R=\bigoplus_{i\geq0}R_i$ be a Gorenstein isolated singularity such that $R_0$ is a field, and such that there exists a triangle equivalence $\sCM^\Z\!R\simeq\Db(\mod A)$ for a finite dimensional algebra $A$. Suppose that the $a$-invariant of $R$ is $0$. Then we have the following.
	\begin{enumerate}
		\item $A$ is semisimple.
		\item There is a triangle equivalence $\sCM R\simeq\Db(\mod A)/F$ for an autoequivalence $F$.
		\item $R$ is of finite representation type and of hereditary representation type.
	\end{enumerate}
\end{Lem}
\begin{proof}
	Since the $a$-invariant of $R$ is $0$, the graded Auslander-Reiten duality implies that $\sCM^\Z\!R$ has Serre functor $[\dim R-1]$, hence $\Db(\mod A)$ is a Calabi-Yau triangulated category. This forces $A$ to be semisimple. Then it is easy (or use Theorem \ref{orbit}) to deduce a triangle equivalence $\sCM R\simeq\Db(\mod A)/F$ for the autoequivalence $F$ of $\Db(\mod A)$ corresponding to the degree shift $(1)$ on $\sCM^\Z\!R$. It follows that $R$ is of finite representation type as well as hereditary representation type.
\end{proof}

\subsection{Consequences of hereditary representation types}
Since hereditary algebras are very special among all finite dimensional algebras, it should be possible to provide various implications on commutative rings using the theory of finite dimensional algebras. We present some of them in this section.
\subsubsection{Gradability}
The first one is gradability of Cohen-Macaulay modules. We can also view this as graded hereditary type implies hereditary type of the completion, an analogue of the corresponding result for finite representation type \cite{AR89c}\cite[Chapter 15]{Yo90}.
\begin{Thm}\label{gradable}
Let $R=\bigoplus_{i\geq0}R_i$ be a positively graded commutative Gorenstein ring with $R_0$ a field, and let $\widehat{R}=\prod_{i\geq0}R_i$ the completion. Suppose that $R$ is of graded hereditary representation type.
\begin{enumerate}
\item $\widehat{R}$ is of hereditary representation type.
\item Every Cohen-Macaulay $\widehat{R}$-module is gradable, that is, the functor $\sCM^\Z\!R/(1)\to\sCM\widehat{R}$ is an equivalence.
\end{enumerate}
\end{Thm}
\begin{proof}
	We know by Proposition \ref{sg's} that $\sCM R\simeq\sCM\widehat{R}$. Also by \cite{haI} the given triangle equivalence $\sCM^\Z\!R\simeq\Db(\mod H)$ induces $\sCM R\simeq(\Db(\mod H)/F)_\triangle$ for the autoequivalence $F$ of $\Db(\mod H)$ corresponding to the degree shift $(1)$ on $\sCM^\Z\!R$, where $(-)_\triangle$ means the canonical triangulated hull \cite{Ke05}. It remains to show that the orbit functor $\Db(\mod H)\to\Db(\mod H)/F$ is dense. Let $M\in\sCM^\Z\!R$ the object corresponding to $H\in\Db(\mod H)$. We only have to verify the conditions in Theorem \ref{orbit}.
	
	(i)  Note that an indecomposable object $X\in\Db(\mod H)$ lies in $\mod H$ if and only if $\Hom_{\rD(H)}(H,X)\neq0$. Therefore we have to show that for any indecomposable $N\in\sCM^\Z\!R$, the space $\sHom_R^\Z(M,N(i))$ vanishes for almost all $i\in\Z$. This follows from that $R$ is an isolated singularity, so the space $\sHom_R(M,N)=\bigoplus_{i\in\Z}\sHom_R^\Z(M,N(i))$ is finite dimensional.
	
	(ii)  Let $d=\dim R$ and $a$ the $a$-invariant of $R$. In view of Lemma \ref{a=0} it is enough to consider the case $a\neq0$. Comparing the Serre functors of $\sCM^\Z\!R\simeq\Db(\mod H)$ yields $(a)[d-1]\leftrightarrow\nu:=-\lotimes_HDH$, thus $(a)\leftrightarrow \nu_{d-1}:=\nu\circ[-d+1]$. Note that it is enough to verify the condition for $F^a$, hence for $\nu_{d-1}$.
	If $d\geq3$, a fundamental domain for the action of $\nu_{d-1}$ on $\Db(\mod H)$ can be taken as $(\mod H)\vee(\mod H)[1]\vee\cdots\vee(\mod H)[d-3]\vee(\proj H)[d-2]$. Therefore any $\nu_{d-1}$-orbit of an indecomposable object intersects in particular with $\mod H\vee(\mod H)[1]\vee\cdots\vee(\mod H)[d-2]$.
	If $d=2$, we have $\nu_{d-1}=\tau$, the Auslander-Reiten translation. Then $\bigoplus_{i\in\Z}\Hom_{\Db(\mod H)}(H,\tau^{-i}H)=\bigoplus_{i\in\Z}\sHom_R^\Z(M,M(ia))$ being finite dimensional implies $H$ is of Dynkin type. Therefore the subcategory $\proj H$ as above is still a fundamental domain for $\nu_{d-1}$, so any $\nu_{d-1}$-orbit intersects with $\mod H$.
	Finally if $d\leq1$, the fundamental domain becomes $\mod H\setminus\inj H$ for $d=1$ and $(\mod H)[-1]\vee\mod H\setminus\inj H$ for $d=0$, where ``$\setminus\inj H$'' means the summands must not lie in $\inj H$, so that we do have the assertion.
\end{proof}

\subsubsection{Uniqueness of enhancements}
Let us turn to the consequences of strictly hereditary representation type. Consider the triangulated category
\[ \T=\Db(\mod H)/\tau^{-1/(d-1)}[1]. \]
Then $P\in\T$ from Remark \ref{remark} is a $d$-cluster tilting object by \cite[Corollary 6.4]{ha4}, and we have $H=\End_\T(\bigoplus_{i=0}^{d-2}P[-i])$. Therefore a $(d+1)$-dimensional Gorenstein ring $R$ of strictly hereditary representation type has a natural $d$-cluster tilting object $T\in\sCM R$ such that $\sEnd_R(\bigoplus_{i=0}^{d-2}\Om^iT)$ is hereditary.

Recall that an {\it enhancement} of a triangulated category $\T$ is a dg category $\A$ such that $\per\A\simeq\T$. A triangulated category has a {\it unique enhancements} if any two enhancements are derived Morita equivalent.
Then one has the uniqueness of enhancements for $\sCM R$ which we repeat from \cite{ha4}.
\begin{Thm}\label{ue}
Let $R$ be a complete Gorenstein local isolated singularity of dimension $\geq3$ over a perfect field which is of strictly hereditary representation type. Then $\sCM R$ has a unique enhancement.
\end{Thm}
\begin{proof}
	A direct consequence of \cite[Theorem 4.19]{ha4}.
\end{proof}

\subsubsection{Number of direct summands of cluster tilting objects}
Let $R$ be a commutative ring of strictly hereditary representation type. We study $d$-cluster tilting objects in $\CM R$ by means of the triangle equivalence
\[ \sCM R\simeq\Db(\mod H)/\tau^{-1/(d-1)}[1]. \]
More precisely, we work on the hereditary algebra side, and employ the correspondence between $d$-cluster tilting objects in the cluster category $\Db(\mod H)/\tau^{-1/(d-1)}[1]$ and silting objects in the derived category $\Db(\mod H)$ \cite{haI2}.
As an application we provide the following result on the structure of the set of $d$-cluster tilting objects in $\CM R$.
Note that this is well-established for $d+1=3$ (e.g.\! \cite{Iy07b,Pa09}), but widely open for higher dimensions.
\begin{Thm}\label{smd}
Let $R$ be a commutative complete Gorenstein local isolated singularity of dimension $d+1\geq3$ which is of strictly hereditary representation type. Then any $d$-cluster tilting object in $\CM R$ has the same number of non-isomorphic indecomposable direct summands.
\end{Thm}
\begin{proof}
	It is enough to consider the stable category. We denote by $F=\tau^{-1/(d-1)}[1]$ the autoequivalence of $\Db(\mod H)$ by which we take the orbit. In \cite{haI2} we shall prove that any $d$-cluster tilting object in $\Db(\mod H)/F$ lifts to an object $T\in\Db(\mod H)$ such that $\bigoplus_{i=0}^{d-2}F^iT$ is a silting object in $\Db(\mod H)$, and $\add F^iT\cap\add F^jT=0$ when $i\neq j$. If we denote the number of direct summands of $X$ by $|X|$, it follows that $|T|=|H|/(d-1)$, which is constant.
\end{proof}

\begin{Rem}\label{bo}
Let us point out that the above result is an evidence of the {\it non-commutative Bondal--Orlov conjecture} \cite{VdB04}; recall that it claims for a commutative Gorenstein normal domain $R$, the endomorphism rings $\End_R(M)$ and $\End_R(N)$ are derived equivalent whenever $M$ and $N$ are NCCRs (e.g.\! cluster tilting objects) for $R$. In particular $M$ and $N$ are expected to have the same number of summands, which we have proved for $R$ of strictly hereditary representation type.
\end{Rem}

\subsection{Variation of the definition}\label{else}
Let us discuss some possible variants for the definition of hereditary representation types. It is based on examples where one has hereditary abelian categories in place of ``$\mod H$ for a finite dimensional algebra $H$''. We do not write down the analogue of Definition \ref{hered}, but give some examples of Gorenstein rings of hereditary representation type in this general sense.

First we give an example from Geigle-Lenzing's weighted projective lines.
\begin{Ex}
Let $R=k[x,y]/(f)$ be a graded $1$-dimensional hypersurface singularity given by either
\[ f=\Prod_{i=1}^4(x-\al_iy),\quad \deg x=\deg y=1 \qquad \text{ or }\qquad f=\Prod_{i=1}^3(x-\al_iy^2),\quad \deg x=2, \deg y=1. \]
We know from \cite[Proposition 2.4]{BIY} that there is a triangle equivalence
\[ \sCM^\Z\!R\simeq\Db(\mod A), \]
for some finite dimensional algebra $A$, which can be taken to be a {canonical algebra} of type $(2,2,2,2)$. Thus, it is presented by the quiver and relations below for some $\la\in k\setminus\{0,1\}$.
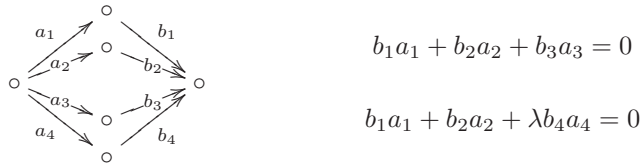
\[ \xymatrix@R=0.5mm{
	&\circ\ar[ddr]^-{b_1}&\\
	&\circ\ar[dr]|-{b_2}& &&b_1a_1+b_2a_2+b_3a_3=0\\
	\circ\ar[uur]^{a_1}\ar[ur]|-{a_2}\ar[dr]|-{a_3}\ar[ddr]_-{a_4}&&\circ\\
	&\circ\ar[ur]|-{b_3}& &&b_1a_1+b_2a_2+\la b_4a_4=0\\
	&\circ\ar[uur]_-{b_4}& } \]
Since $A$ is derived equivalent to $\coh X$ for a weighted projective line $X$ of Geigle-Lenzing, we have a triangle equivalence
\[ \sCM^\Z\!R\simeq\Db(\coh X) \]
for the hereditary abelian category $\coh X$. Therefore $R$ is of ``graded hereditary representation type''.
\end{Ex}

The second example is an interpretation of a Gorenstein ring of countable representation type in this context.
\begin{Ex}
Consider the ring
\[ R=k[[x,y]]/(y^2), \]
known as a Cohen-Macaulay ring of countable representation type, see \cite[Chapter 14]{LW}. We give a grading on $R$ by $\deg x=0$ and $\deg y=1$. Then by Minamoto--Yamaura's theorem \cite{MY} we have a triangle equivalence
\[ \sCM^\Z\!R\simeq\Db(\mod k[[x]]). \]
Indeed, $R$ is homologically well-graded in the sense of \cite[Definition 4.4, Theorem 5.3]{MY}, hence the canonical functor $\Db(\mod k[[x]])\to\Db(\mod^\Z\!R)\twoheadrightarrow\sCM^\Z\!R$ (called {\it Happel's functor} in \cite{MY}) is an equivalence by \cite[Theorem 1.3]{MY}. Since $k[[x]]$ is hereditary, we can view $R$ is of ``graded hereditary representation type''. Since the objects of $\Db(\mod k[[x]])$ are easily classified, we deduce the classification of objects in $\sCM^\Z\!R$.
\end{Ex}

Let us also note the following application of Orlov's theorem. 
\begin{Ex}
Let $S=k[x,y,z]$ which we consider as a graded ring with $\deg x=\deg y=\deg z=1$, and let $f\in S$ be a homogeneous element of degree $3$ such that $R=S/(f)$ is an isolated singularity. Then the corresponding projective variety $E=\Proj R$ is a (smooth) elliptic curve. Since the graded ring $R$ has $a$-invariant $0$, Orlov's theorem \cite{Or09a} gives a triangle equivalence
\[ \sCM^\Z\!R \simeq \Db(\coh E). \]
Since $\coh E$ is a hereditary abelian category, we can say that $R$ is of ``graded hereditary representation type''.
\end{Ex}

\section{Applications of the Morita-type theorem}\label{MORITA}
We give (non-trivial) examples of commutative Gorenstein rings of strictly hereditary representation type.
One of the strongest results to give equivalences of categories is Morita-type theorem. We restate Theorem \ref{jus} for small dimensions, with some additional claims on the endomorphism rings.
\begin{Thm}\label{morita}
Let $\T$ be an algebraic $d$-CY triangulated category with a $d$-cluster tilting object $T$.
\begin{enumerate}
	\item\cite{KRac} Suppose $d=2$ and $H=\End_\T(T)$ is hereditary. Then there exists a triangle equivalence
	\[ \T\simeq\rC_2(H). \]
	\item\cite[Theorem 1.2, 1.3(2)]{ha4} Suppose that $d=3$ and $\End_\T(T)$ is hereditary. Then $H=\End_\T(T\oplus T[-1])$ is hereditary and there exists a triangle equivalence
	\[ \T\simeq\rC_3^{(1/2)}(H) \]
	provided $H$ is $1$-representation infinite.
	\item\cite[Theorem 1.2, 1.3(3)]{ha4} Suppose that $d=4$ and $\End_\T(T)$ is semisimple. Then $H=\End_\T(T\oplus T[-1]\oplus T[-2])$ is hereditary, and there exists a triangle equivalence
	\[ \T\simeq\rC_4^{(1/3)}(H) \]
	provided $H$ is $1$-representation infinite.
\end{enumerate}
\end{Thm}
Moreover, for the cases $d=3, 4$, the following observation tells how to compute the larger endomorphism algebra $H=\End_\T(T\oplus\cdots\oplus T[-(d-2)])$ from the smaller one $\End_\T(T)$ in terms of higher almost split sequences.
\begin{Thm}[{\cite[Proposition 2.13, 2.22]{ha4}}]\label{end}
Let $\T$ be a $d$-CY triangulated category with a $d$-cluster tilting object $T$.
\begin{enumerate}
\item Suppose that $d=3$ and $\End_\T(T)=kQ$ for a finite acyclic quiver $Q$. Consider the $3$-almost split sequence
\[ \xymatrix@!C=3mm@!R=2mm{
	&\disoplus_{i\to j}T_j\ar[dr]\ar[rr]&&A_i\ar[dr]\ar[rr]&&\disoplus_{j\to i}T_j\ar[dr]&\\
	T_i\ar[ur]&&\bullet\ar[ur]&&\bullet\ar[ur]&&T_i } \]
at each summand $T_i$ of $T$.
Then $\End_\T(T\oplus T[-1])=k\widetilde{Q}$ for the quiver $\widetilde{Q}$ constructed as follows:
\begin{itemize}
	\item Take two copies $Q^{(1)}$ and $Q^{(2)}$ of $Q$.
	\item Add $n_{ij}$ arrows from $i\in Q^{(1)}_0$ to $j\in Q^{(2)}_0$, where $n_{ij}$ is the number of direct summands $T_j$ in $A_i$.
\end{itemize}
\item Suppose that $d=4$ and $\End_\T(T)=\prod_{i=1}^nk$, the direct product $n$ copies of $k$. Consider the $4$-almost split sequence
\[ \xymatrix@R=2mm@!C=5mm{
	&0\ar[dr]\ar[rr]&&A_i\ar[dr]\ar[rr]&&B_i\ar[dr]\ar[rr]&&0\ar[dr]& \\
	T_i\ar[ur]&&T_i[1]\ar[ur]&&\bullet\ar[ur]&&T_i[-1]\ar[ur]&&T_i } \]
at each summand $T_i$ of $T$.
Then $\End_\T(T\oplus T[-1]\oplus T[-2])=k\widetilde{Q}$ for the quiver $\widetilde{Q}$ constructed as follows: 
\begin{itemize}
	\item It has $3n$ vertices $\{(i,a)\mid 1\leq i\leq n,\, a=1,2,3\}$.
	\item Add $m_{ij}$ arrows from $(i,1)$ to $(j,2)$, from $(i,2)$ to $(j,3)$, and from $(j,1)$ to $(i,3)$, where $m_{ij}$ is the number of direct summands $T_i$ in $B_j$.
\end{itemize}
\end{enumerate}
\end{Thm}
\begin{Rem}
In (1) above we have $n_{ij}=n_{ji}$ (\cite[Proposition 2.13]{ha4}). Similarly in (2), we have that $m_{ij}$ is equal to the number of summands $T_i$ in $B_j$ (\cite[Proposition 2.22]{ha4}).
\end{Rem}

We give examples of commutative Gorenstein rings to which the above theorem can be applied. We present one example for each dimension $3$, $4$ and $5$.
\subsection{Dimension $3$}
We give a $3$-dimensional commutative Gorenstein ring whose stable category of Cohen-Macaulay modules have a $2$-cluster tilting object with hereditary endomorphism ring, so that Theorem \ref{morita}(1) can be applied. Let
\[ R=k[[x,y]]^{(2)}\seg k[[u,v]]^{(2)} \]
be the (complete) Segre product of second Veronese subrings, where we give a complete $\Z$-grading on $k[[x,y]]^{(2)}$ whose degree $n$ part consists of polynomials in $x$, $y$ of degree $2n$, and the same for $k[[u,v]]$. Our result is that this ring is of hereditary representation type.
\begin{Thm}
	Let $R=k[[x,y]]^{(2)}\seg k[[u,v]]^{(2)}$ be the Segre product of Veronese subrings of polynomial rings. Then we have a triangle equivalence
	\[ \sCM R\simeq\rC_2(kQ) \]
	for $Q=\xymatrix{\circ\ar@2[r]&\circ&\circ\ar@2[l]}$.
\end{Thm}
\begin{proof}
	It is classical that the $2$-dimensional $A_1$-singularity $T=k[[x,y]]^{(2)}$ is of finite representation type with an additive generator (in other words, $1$-cluster tilting object) $k[[x,y]]=T\oplus N$, where $N$ is consists of power series of odd degrees. Its endomorphism ring is presented by the quiver below with commutativity relations.
	
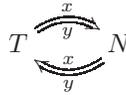
\[ \xymatrix{ T\ar@2@/^10pt/[r]^-x_-y&N\ar@2@/^10pt/[l]_-x^-y} \]
	We give a complete grading on $N$ by letting its degree $n$ part to be the polynomials of degree $2n+1$, which induces a positive grading on the endomorphism ring $\End_T(T\oplus N)$. In terms of the above quiver, it is given by the rightgoing arrows to have degree $0$, and the leftgoing arrows to have degree $1$. We give the same grading for the corresponding objects for $k[[u,v]]$.
	
	Then by Theorem \ref{CT} the Segre product $R=k[[x,y]]^{(2)}\seg k[[u,v]]^{(2)}$ has a $2$-cluster tilting object whose endomorphism ring is the ($1$-)Segre product of two copies of $\End_T(T\oplus N)$. It is presented by the quiver below with commutativity relations (where the diagonal arrows represent $xu, yu, xv, yv$).
	
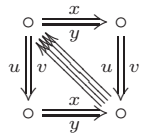
\[ \xymatrix{
		\circ\ar@2[r]^-x_-y\ar@2[d]_-u^-v&\circ\ar@2[d]_-u^-v\\
		\circ\ar@2[r]^-x_-y&\circ\ar@<0.8ex>[ul]\ar@<0.4ex>[ul]\ar@<0.0ex>[ul]\ar@<-0.4ex>[ul] } \]
	It follows that the stable endomorphism ring is the path algebra of the quiver $\xymatrix{\circ\ar@2[r]^-x_-y&\circ&\circ\ar@2[l]_-u^-v}$ as described above. Then Theorem \ref{morita}(1) applies and gives the result.
\end{proof}

\subsection{Dimension $4$}
Let $S=k[[x,y,z,u,v]]$ be the power series ring in $5$ variables, which we view as a complete $\Z^2$-grading by $\deg x=\deg y=\deg z=(1,0)$, $\deg u=(0,1)$, and $\deg v=(0,2)$. The singularity we are interested in is the Veronese subring
\[ R=S^{(1,1)}. \]
Thus, $R$ is isomorphic to the Segre product of $S^\prime=k[[x,y,z]]$ and $S^\pprime=k[[u,v]]$ with $\deg x=\deg y=\deg z=1$ and $\deg u=1$, $\deg v=2$. This is a $4$-dimensional Gorenstein isolated singularity. For each $i\in\Z$ we define the (completely graded) $R$-module $M_i$ by
\[ M_i=\Prod_{j\in\Z}S_{(i,0)+(j,j)}. \]
We can also write it as a Segre product $M_i=S^\prime(i)\seg S^\pprime$. It is well-known (e.g.\! use Lemma \ref{cov}) that each $M_i$ is reflexive, but only a small part is Cohen-Macaulay. In fact, in this case the Cohen-Macaulay modules among the $M_i$'s are $M_0=R$, $M_{\pm1}$, and $M_{\pm2}$.
\begin{Thm}\label{3CY}
Let $R=S^{(1,1)}$ and $M_i\in\mod R$ be as above.
\begin{enumerate}
	\item The $R$-module $T:=M_{-1}\oplus R\oplus M_1\in\CM R$ is $3$-cluster tilting.
	\item The stable endomorphism ring $\sEnd_R(T)$ is isomorphic to the path algebra of type $A_2$.
	\item We have that $\sEnd_R(T\oplus T[-1])$ is isomorphic to the path algebra of the following quiver $Q$.
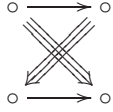
	\[ \xymatrix{
		\circ\ar[r]\ar@3[dr]&\circ\ar@3[dl]\\
		\circ\ar[r]&\circ} \]
\end{enumerate}
Therefore $R$ is of strictly hereditary representation type, and there exists a triangle equivalence $\sCM R\simeq \rC_3^{(1/2)}(kQ)$.
\end{Thm}
\begin{proof}
	By Theorem \ref{CT} we know that $R\oplus M_1\oplus M_2\in\CM R$ is $3$-cluster tilting, hence so is $T=M_{-1}\oplus R\oplus M_{1}$. Its endomorphism ring is presented by the following quiver with commutativity relations.
	
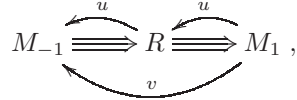
\[ \xymatrix{ M_{-1}\ar@3[r]&R\ar@3[r]\ar@/_10pt/[l]_-u&M_1\ar@/_10pt/[l]_-u\ar@/^20pt/[ll]_-v}, \]
	where the triple arrows represent $x$, $y$, and $z$. It follows that the stable endomorphism ring $\sEnd_R(T)$ is
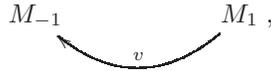
	\[ \xymatrix{ M_{-1}&& M_1\ar@/^20pt/[ll]_-v }, \]
	just the path algebra of type $A_2$.
	In Proposition \ref{3AR} below we will compute the higher almost split sequences in $\add T$. Together with Theorem \ref{end}(1), we deduce that $\sEnd_R(T\oplus T[-1])$ is as described, so that Theorem \ref{morita}(2) gives the desired equivalence.
\end{proof}

\begin{Prop}\label{3AR}
The $3$-almost split sequences in $\add T\in\CM R$ are as follows.
\[ \xymatrix{
	0\ar[r]&R\ar[r]&M_{1}^3\oplus M_{-1}\ar[r]&R^6\ar[r]&M_1\oplus M_{-1}^3\ar[r]&R\\
	0\ar[r]&M_1\ar[r]& R\oplus M_{-1}\ar[r]&M_{-1}^3\ar[r]&R^3\ar[r]&M_1\ar[r]&0\\
	0\ar[r]&M_{-1}\ar[r]&R^3\ar[r]&M_1^3\ar[r]&R\oplus M_1\ar[r]&M_{-1}\ar[r]&0 } \]
\end{Prop}
\begin{proof}
	We follow the recipe of Section \ref{daiji}. Let $S^\prime=k[[x,y,z]]$ and $S^\pprime=k[[u,v]]$, which we regard as a completely graded ring with $\deg x=\deg y=\deg z=1$ and $\deg u=1, \deg v=2$. Let $S=k[[x,y,z,u,v]]$ be the completion of $S^\prime\otimes S^\pprime$, which we view as a completely $\Z^2$-graded ring. First form the Koszul complexes
	\[ \xymatrix@R=1mm{
		0\ar[r]&S^\prime(-3)\ar[r]&S^\prime(-2)^3\ar[r]&S^\prime(-1)^3\ar[r]&S^\prime\ar[r]&0\\
		&0\ar[r]&S^\pprime(-3)\ar[r]&S^\pprime(-2)\oplus S^\pprime(-1)\ar[r]&S^\pprime\ar[r]&0 } \]
	for $S^\prime$ and $S^\pprime$.
	
	Let us now compute the $3$-fundamental sequence at $R$. We view the Koszul complexes as the mapping cones of the morphisms of complexes below. 
	\[
	\xymatrix{
		X^\prime\ar[d]\ar@{}[r]|<\colon&S^\prime(-3)\ar[r]&S^\prime(-2)^3\ar[d]&\\
		Y^\prime\ar@{}[r]|<\colon&&S^\prime(-1)^3\ar[r]&S^\prime,}
	\qquad
	\xymatrix{
		X^\pprime\ar[d]\ar@{}[r]|<\colon&S^\pprime(-3)\ar[r]\ar[d]&S^\pprime(-2)\ar[d]\\
		Y^\pprime\ar@{}[r]|<\colon&S^\pprime(-1)\ar[r]&S^\pprime}
	\]
	Then (the completion of) their tensor product $X^\prime\otimes X^\pprime \to Y^\prime\otimes Y^\pprime$ is
	\[ \xymatrix{
		X^\prime\otimes X^\pprime\ar[d]\ar@{}[r]|<\colon&S(-3,-3)\ar[r]&S(-2,-3)^3\oplus S(-3,-2)\ar[d]\ar[r]&S(-2,-2)^3\ar[d]\\
		Y^\prime\otimes Y^\pprime\ar@{}[r]|<\colon&&S(-1,-1)^3\ar[r]&S(0,-1)\oplus S(-1,0)^3\ar[r]&S} \]
	for $S=k[[x,y,z,u,v]]$ regarded as a completely $\Z^2$-graded ring with $\deg x=\deg y=\deg z=(1,0)$ and $\deg u=(0,1), \deg v=(0,2)$. Taking the degree $\{(j,j)\mid j\in\Z\}$-part of the mapping cone yields a complex
	\[ \xymatrix{	0\ar[r]&R(-3)\ar[r]&M_1(-3)^3\oplus M_{-1}(-2)\ar[r]&R(-2)^3\oplus R(-1)^3\ar[r]&M_1(-1)\oplus M_{-1}^3\ar[r]&R}, \]
	which is the $3$-fundamental sequence by Proposition \ref{resolution}.
	
	Similarly for the almost split sequence at $M_1$, we take the morphisms of complexes as below.
	\[ 
	\xymatrix{
		S^\prime(-2)\ar[d]&&\\
		S^\prime(-1)^3\ar[r]&{S^\prime}^3\ar[r]&S^\prime(1),}
	\qquad
	\xymatrix{
		S^\pprime(-3)\ar[r]&S^\pprime(-2)\oplus S^\pprime(-1)\ar[d]\\
		&S^\pprime}
	\]
	Take the tensor product and its diagonal part, we obtain the desired almost split sequence at $M_1$.
	
	Finally for $M_{-1}$, we consider the following morphisms
	\[ 
	\xymatrix{
		S^\prime(-4)\ar[r]&S^\prime(-3)^3\ar[r]&S^\prime(-2)^3\ar[d]\\
		&&S^\prime(-1),}
	\qquad
	\xymatrix{
		S^\pprime(-3)\ar[d]&\\
		S^\pprime(-2)\oplus S^\pprime(-1)\ar[r]&S^\pprime,} \]
	which yield the almost split sequence at $M_{-1}$.
\end{proof}

\subsection{Dimension $5$}
Let $S=k[[x_0,x_1,x_2,y_0,y_1,y_2]]$ be the power series ring in $6$ variables, on which we give a complete grading by $\deg x_i=(1,0)$ and $\deg y_i=(0,1)$. We let
\[ R=S^{(1,1)} \]
the Veronese subring. This is a $5$-dimensional complete Gorenstein local isolated singularity, which is also isomorphic to the complete Segre product of two copies of power series rings in three variables.
For each $i\in\Z$ we define the (completely graded) $R$-module $M_i$ by
\[ M_i=\Prod_{j\in\Z}S_{(i,0)+(j,j)}. \]
As in the previous subsection, we have $M_i=S^\prime(i)\seg S^\pprime$ as completely graded $R$-modules, and the Cohen-Macaulay modules among the $M_i$'s are $M_0=R$, $M_{\pm1}$, and $M_{\pm2}$.
\begin{Thm}\label{4CY}
Let $R=S^{(1,1)}$ and $M_i\in\mod R$ be as above.
\begin{enumerate}
	\item The $R$-module $T:=R\oplus M_1\oplus M_{-1}\in\CM R$ is $4$-cluster tilting.
	\item The stable endomorphism ring is $\sEnd_R(T)=k\times k$, which is semisimple.
	\item We have $\sEnd_R(T\oplus T[-1]\oplus T[-2])$ is isomorphic to the path algebra of the following $\widetilde{A_5}$-quiver $Q$ with triple arrows.
	
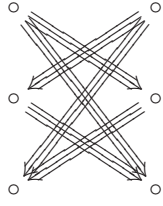
\[ \xymatrix@C=15mm{
		\circ\ar@3[dr]\ar@3[ddr]&\circ\ar@3[dl]\ar@3[ddl]\\
		\circ\ar@3[dr]&\circ\ar@3[dl]\\
		\circ&\circ} \]
\end{enumerate}
Therefore, $R$ is of strictly hereditary representation type, and there exists a triangle equivalence $\sCM R\simeq\rC_4^{(1/3)}(kQ)$.
\end{Thm}
\begin{proof}
	We know by Theorem \ref{CT} that $R\oplus M_1\oplus M_2\in\CM R$ is a $4$-cluster tilting object, hence so is $T=R\oplus M_1\oplus M_{-1}$. Its endomorphism ring is presented by the quiver below with commutativity relations, where the rightgoing arrows are labelled by $x_0, x_1, x_2$, and the leftgoing arrows by $y_0, y_1, y_2$.
	\[ \xymatrix{ M_{-1}\ar@3@/^6pt/[r]&R\ar@3@/^6pt/[r]\ar@3@/^6pt/[l]&M_1\ar@3@/^6pt/[l] } \] 
	It follows that the stable endomorphism ring $\sEnd_R(M)$ is isomorphic to $k\times k$. We then obtain the quiver of $\sEnd_R(T\oplus T[-1]\oplus T[-2])$ by Theorem \ref{end} and Proposition \ref{4AR} below. Now Theorem \ref{morita}(3) applies.
\end{proof}

As in the $3$-dimensional case, we use the computation of higher Auslander-Reiten sequences. 
\begin{Prop}\label{4AR}
The $4$-almost split sequences for $R, M_{\pm1}\in\CM R$ are the following.
\[ \xymatrix@R=1mm@!C=12mm{
	0\ar[r]&R\ar[r]&M_1^{3}\oplus M_{-1}^{3}\ar[r]&R^{9}\ar[r]&R^{9}\ar[r]&M_1^{3}\oplus M_{-1}^{3}\ar[r]&R\\
	0\ar[r]&M_1\ar[r]&R^{3}\ar[r]&M_{-1}^3\ar[r]&M_{-1}^3\ar[r]&R^3\ar[r]&M_1\ar[r]&0\\
	0\ar[r]&M_{-1}\ar[r]&R^{3}\ar[r]&M_{1}^3\ar[r]&M_{1}^3\ar[r]&R^3\ar[r]&M_{-1}\ar[r]&0 } \]
\end{Prop}
\begin{proof}
	Let $S^\prime=k[[x_0,x_1,x_2]]$ and $S^\pprime=k[[y_0,y_1,y_2]]$ be completely graded power series rings with standard gradings. Recall that our construction says $M_i=S^\prime(i)\seg S^\pprime$, that is, the ``shifted diagonal'' degree $\{(i+j,j)\mid j\in\Z\}$-part of the $\Z^2$-graded ring $S^\prime\seg S^\pprime=k[[x_0,x_1,x_2,y_0,y_1,y_2]]$.
	We first prepare the Koszul complexes
	\[ \xymatrix@R=1mm{
		0\ar[r]&S^\prime(-3)\ar[r]&S^\prime(-2)^3\ar[r]&S^\prime(-1)^3\ar[r]&S^\prime\ar[r]&k\ar[r]&0\\
		0\ar[r]&S^\pprime(-3)\ar[r]&S^\pprime(-2)^3\ar[r]&S^\pprime(-1)^3\ar[r]&S^\pprime\ar[r]&k\ar[r]&0 } \]
	for $S^\prime$ and $S^\pprime$.
	
	We first construct the $4$-fundamental sequence at $R$. For this we view the Koszul complexes as mapping cones of the morphisms of complexes below.
	\[
	\xymatrix{
		X^\prime\ar[d]\ar@{}[r]|<\colon&S^\prime(-3)\ar[r]&S^\prime(-2)^3\ar[d]&\\
		Y^\prime\ar@{}[r]|<\colon&&S^\prime(-1)^3\ar[r]&S^\prime,}
	\qquad
	\xymatrix{
		X^\pprime\ar[d]\ar@{}[r]|<\colon&S^\pprime(-3)\ar[r]&S^\pprime(-2)^3\ar[d]&\\
		Y^\pprime\ar@{}[r]|<\colon&&S^\pprime(-1)^3\ar[r]&S^\pprime}
	\]
	Taking the tensor product $X^\prime\otimes X^\pprime\to Y^\prime\otimes Y^\pprime$, we get a morphism of complexes
	\[ \xymatrix{
		S(-3,-3)\ar[r]&S(-2,-3)^3\oplus S(-3,-2)^3\ar[r]&S(-2,-2)^9\ar[d]&&\\
		&&S(-1,-1)^9\ar[r]&S(0,-1)^3\oplus S(-1,0)^3\ar[r]& S} \]
	of completely $\Z^2$-graded $S$-modules. By Proposition \ref{resolution}, the diagonal degree part of the mapping cone of the above morphism of complexes yields an exact sequence
	\[ \xymatrix{0\ar[r]&R(-3)\ar[r]&M_1(-3)^{3}\oplus M_{-1}(-2)^{3}\ar[r]&R(-2)^{9}\ar[r]&R(-1)^{9}\ar[r]&M_1(-1)^{3}\oplus M_{-1}^{3}\ar[r]&R} \]
	of completely $\Z$-graded $R$-modules, which is the $4$-fundamental sequence at $R$.
	
	For the $4$-almost split sequence at $M_1$, consider the following morphisms of complexes as $X^\prime\to Y^\prime$ and $X^\pprime\to Y^\pprime$.
	\[
	\xymatrix{
		S^\prime(-2)\ar[r]&S^\prime(-1)^3\ar[r]&{S^\prime}^3\ar[d]\\
		&&S^\prime(1),}
	\qquad
	\xymatrix{
		S^\pprime(-3)\ar[d]&&\\
		S^\pprime(-2)^3\ar[r]&S^\pprime(-1)^3\ar[r]&S^\pprime}
	\]
	Taking the tensor product and its diagonal degree part, we obtain the desired sequence.
	
	Finally, the $4$-almost split sequence at $M_{-1}$ is the $R$-dual of that for $M_1$.
\end{proof}

\section{Extended numerical semigroup rings and tilting theory}
The aim of this section is to give a family of graded Gorenstein rings of (graded) hereditary representation type using tilting theory. It is given as {\it extended numerical semigroup rings} which we introduce as a generalization of classical numerical semigroup rings \cite{RS}.

\subsection{Extended numerical semigroup rings}
Throughout this section we denote by $\NN=\{0,1,2,\ldots\}$ the set of non-negative integers and $L$ a finite abelian group, which we write multiplicatively.
\begin{Def}\label{exnum}
An {\it extended numerical semigroup} is a submonoid $S\subset\NN\times L$ (containing the unit) whose complement is finite. For an arbitrary field $k$, we call the corresponding semigroup algebra $k[S]$ or its completion $k[[S]]$ the {\it extended numerical semigroup algebra} associated to $S$.
\end{Def}
We shall use the identifications $k[\NN]=k[t]$ and $k[\NN\times L]=k[t]\otimes k[L]$, and write its elements as $t^n\otimes \la$ for $(n,\la)\in \NN\times L$. Since $S\subset \NN\times L$ has finite complement, there exists $N\in\NN$ such that $(n,\la)\in S$ for all $n>N$ and $\la\in L$.
We also consider a $\Z$-grading on $k[t]\otimes k[L]$ by setting $\deg (t^n\otimes \la)=n$ for $(n,\la)\in\NN\times L$. Then $R=k[S]$ inherits a $\Z$-grading.
\begin{Prop}\label{shoutai}
	Let $k$ be an arbitrary field, $S\subset\NN\times L$ an extended numerical semigroup, $R=k[S]$, and $\widehat{R}=k[[S]]$.
	\begin{enumerate}
		\item\label{red} $R$ (resp. $\widehat{R}$) is reduced if and only if the order of $L$ is invertible in $k$.
		\item\label{K} The graded total quotient ring $K$ of $R$ is $k[\Z\times L]=k[t^{\pm1}]\otimes k[L]$.
		\item The total quotient ring $\widehat{K}$ of $\widehat{R}$ is $k[[\Z\times L]]=k((t))\otimes k[L]$. 
	\end{enumerate}
\end{Prop}
\begin{proof}
	(\ref{red})  Noting that $R$ is a subset of $k[\NN\times L]=k[t]\otimes k[L]$, it is easy to see that $R$ is reduced if and only if $k[L]$ is semisimple. The statement for $\widehat{R}$ is similar.\\
	(\ref{K})  We have a map $R\subset k[t]\otimes k[L]\hookrightarrow k[t^{\pm1}]\otimes k[L]$. Since a homogeneous element of degree $n$ in $k[t]\otimes k[L]$ is of the form $t^n\otimes \mu$ for some $\mu\in k[L]$, it is a non-zero divisor if and only if $\mu\in k[L]$ is a non-zero divisor, or equivalently a unit. Therefore, we see that the map takes every homogeneous non-zero divisor to a unit, inducing a map $K\to k[t^{\pm}]\otimes k[L]$.
	\[ \xymatrix@!R=3mm{
		R\ar[d]\ar[r]&k[t]\otimes k[L]\ar[d]\\
		K\ar@{-->}[r]&k[t^{\pm1}]\otimes k[L]} \]
	We prove that this induced map is an isomorphism. Since the map $R\to k[t^{\pm1}]\otimes k[L]$ is injective and $R\to K$ is essential, we see that $K\to k[t^{\pm1}]\otimes k[L]$ is injective. To prove surjectivity, pick $t^n\otimes \mu\in k[t^{\pm1}]\otimes k[L]$. Then we have $t^n\otimes \mu=(t^{n+m}\otimes \mu)/(t^m\otimes1)$, in which both the numerator and the denominator lie in $R$ for sufficiently large $m$.\\
	(3)  Left to the reader.
\end{proof}

Let us prepare some combinatorial notions which generalize the classical ones.
\begin{Def}
Let $S\subset\NN\times L$ be an extended numerical semigroup.
\begin{enumerate}
	\item We say $S$ is {\it connected} if $(0,\la)\not\in S$ for $\la\neq1$.
	\item The {\it Frobenius number} of $S$ is maximum integer $a\in\NN$ such that $(a,\la)\not\in S$ for some $\la\in L$. We formally put $a=-1$ if $S=\NN\times L$.
	\item We say $S$ is {\it twisted symmetric} if there exists $\tau\in L$ such that $(\Z\times L)\setminus S=(a,\tau)-S$, where $a$ is the Frobenius number of $S$ and $(a,\tau)-S=\{(a-n,\tau\la^{-1}\mid (n,\la)\in S\}$.
\end{enumerate}
\end{Def}

%

It is easy to see that extended numerical semigroup rings are equi-codimensional Cohen-Macaulay of dimension $1$. This can be seen, for example, as follows: Pick $c\in\NN$ such that for every $n\geq c$ and $\la\in L$ one has $(n,\la)\in S$. Then $R$ is a finitely generated free module over $k[t^c]$, hence the assertion follows.

Now we describe the canonical module, compute the $a$-invariant, and characterize when $R$ is Gorenstein. This generalizes the well-known results for classical numerical semigroup rings. We denote by $D\colon M=\bigoplus_{i\in\Z}M_i\mapsto\bigoplus_{i\in\Z}\Hom_k(M_{-i},k)$ the graded dual.
\begin{Prop}\label{omega}
\begin{enumerate}
	\item\label{can} The $R$-module $\om:=D(K/R)$ is a canonical module for $R$.
	\item\label{Frob} The $a$-invariant of $R$ is equal to the Frobenius number of $S$.
	\item\label{sym} $R$ is Gorenstein if and only if $S$ is twisted symmetric.
\end{enumerate}
\end{Prop}
\begin{proof}
	(\ref{can})  Let $c$ be an integer larger than the Frobenius number of $S$. Let $T=k[t^c]$ and we prove that $\Hom_T(R,T)(-c)\simeq\om$ in $\Mod^\Z\!R$, which shows that $\om$ is a (global) canonical module.
	
	For each $\la\in L$ let $C_\la:=\{ i\in \NN\mid (i,\la)\in S \text{ and } (i-c,\la)\not\in S\}$. It is a subset of the interval $[0,2c-1]$, in particular finite. Then $R$ is a free $T$-module with basis $\bigcup_{\la\in L}\{ t^i\otimes\la \mid i\in C_\la\}$, thus $\Hom_T(R,T)$ is free with basis $\bigcup_{\la\in L}\{ t^{-i}\otimes\la^{-1} \mid i\in C_\la\}$.
	On the other hand, the space $K/R$, as an $T$-module, is isomorphic to the direct sum of $DT$ with socle $\bigcup_{\la\in L}\{(t^i,\la)\mid i+c\in C_\la\}$, so $D(K/R)$ is a free $T$-module with basis $\bigcup_{\la\in L}\{t^{-i}\otimes\la^{-1}\mid i+c\in C_\la\}$.
	It follows that multiplication by $t^c$ gives an isomorphism $\Hom_T(R,T)(-c)\xsimeq D(K/R)$. Since everything can be embedded in the $R$-module $K$, it is clear that this isomorphism is $R$-linear.
	

	(\ref{Frob})  The claimed isomorphism $\Hom_T(R,T(-c))\simeq\om$ in $\Mod^\Z\!R$ shows that $\om$ is the graded canonical module for $R$. Since the $a$-invariant is characterized as the minus of the degree of the top of $\om$, it is clear from $\om\simeq D(K/R)$ that it must be the Frobenius number.
	
	(\ref{sym})  Let $a$ be the Frobenius number of $S$, and we denote $\Om:=-(\Z\times L)\setminus S$. 
	Suppose first that $S$ is twisted symmetric and pick $\tau\in L$ such that $(\Z\times L)\setminus S=(a,\tau)-S$, thus $\Om=S-(a,\tau)$. Then multiplication by $t^a\otimes\tau$ induces an isomorphism $\om=D(K/R)\xsimeq R(a)$, hence $R$ is Gorenstein.
	
	Now we prove the converse. If $R$ is Gorenstein we must have an isomorphism $\varphi\colon R\xsimeq \om(-a)$ in $\mod^\Z\!R$. We show that modifying $\varphi$ if necessary, it is given as the multiplication by $t^{-a}\otimes\tau^{-1}$ for some $\tau\in L$, which will yield $S-(a,\tau)=-(\Z\times L)\setminus S$. Let $M=\{\la\in L\mid (0,\la)\in S\}$ and $A:=(\{a\}\times L)\setminus S$, which is closed under the action of $M$. Taking the degree $0$ part of the isomorphism $R\simeq\om(-a)$ we have an isomorphism $\varphi_0\colon k[M]\simeq k[-A]$ of $k[M]$-modules. Now note that any two generators of a rank $1$ free module over a finite dimensional algebra $\La$ transfer to each other by a multiplication by an element of $\La^\times$. Therefore we may assume, modifying $\varphi$ by $\End_R^\Z(R)^\times=R_0^\times=k[M]^\times$ if necessary, that $\varphi_0$ takes $1\in k[M]$ to $t^{-a}\otimes \tau^{-1}$ for some $(a,\tau)\in A$, since for any $(a,\tau)\in A$ the corresponding element $t^{-a}\otimes\tau^{-1}\in k[-A]$ generates $k[-A]$ as a $k[M]$-module. Now $\varphi$ is given as multiplication by $t^{-a}\otimes\tau^{-1}$.
\end{proof}
We also record without proof the following version for the complete setting. This time we use $D$ for the Matlis dual.
\begin{Prop}
Let $\widehat{R}=k[[S]]$ be an extended numerical semigroup ring.
\begin{enumerate}
	\item The $\widehat{R}$-module $\widehat{\om}=D(\widehat{K}/\widehat{R})$ is a canonical module for $\widehat{R}$.
	\item $\widehat{R}$ is Gorenstein if and only if $S$ is twisted symmetric.
\end{enumerate}
\end{Prop}

We take a look at a quick example of extended numerical semigroup rings.
\begin{Ex}
Let $p$ and $q$ be positive integers and consider the hypersurface
\[ R=k[x,y]/(x^p-y^q). \]
It is well-known that if $p$ and $q$ are relatively prime, then $R$ is a domain, and is isomorphic to the (usual) numerical semigroup ring $k[S]$ with $S=\langle p,q\rangle$, the numerical semigroup generated by $p$ and $q$.

We observe that even if $p$ and $q$ are not relatively prime, we can still describe $R$ as an extended numerical semigroup ring. Let $d$ be the gratest common divisor of $p$ and $q$, and put $a=p/d$, $b=q/d$. Since $a$ and $b$ are relatively prime, there exist integers $i,j$ such that $ia-jb=1$ in $\Z/d\Z$.
Now consider the extended numerical semigroup $S\subset\NN\times L$ with $L=\Z/d\Z=\langle \la\rangle$ generated by $(b,\la^i)$ and $(a,\la^j)$.
For example, if $p=6$ and $q=4$ then $d=2$ and one can take $i=j=1$. In this case $S$ can be depicted as follows.
\[ S=
\begin{tabular}{c|ccccccccccccc}
	&$0$&$1$&$2$&$3$&$4$&$5$&$6$&$7$&$8$&$9$&$10$&$11$&$12$\\
	\hline
	$1$&$\bullet$&&&&$\bullet$&$\bullet$&$\bullet$&&$\bullet$&$\bullet$&$\bullet$&$\bullet$&$\bullet$\\
	$\la$&&&$\bullet$&$\bullet$&&&$\bullet$&$\bullet$&$\bullet$&$\bullet$&$\bullet$&$\bullet$&$\bullet$
\end{tabular}
\]
Then the map
\[ k[x,y]/(x^p-y^q)\to k[S], \qquad x\mapsto t^b\otimes\la^i, y\mapsto t^a\otimes\la^j \]
is an isomorphism. 
\end{Ex}

\subsection{An extended numerical semigroup ring of hereditary representation type}
We are now able to give a family of $1$-dimensional Gorenstein rings of hereditary representation type using extended numerical semigroup rings. For a finite dimensional algebra $A$ and each $n\in\Z$, we denote by $\rC_n(A)$ the {\it $n$-cluster category} of $A$, the triangulated hull of $\Db(\mod A)/-\lotimes_ADA[-n]$.
\begin{Thm}\label{num}
Let $L$ be a finite abelian group of order $n$, and let $k$ be a field of characteristic not dividing $n$. Pick $1\neq\la\in L$ and let ${R}=k[[S]]$ be the extended numerical semigroup ring corresponding to
\[ S=(\NN\times L)\setminus \left(\{(0,\mu)\mid \mu\neq1\}\cup \{(1,\la)\}\right). \]
Then there exists a triangle equivalence below for the $n$-subspace quiver $Q$.
\[  \xymatrix@R=1mm{
	&   & &&0\ar[ddll]\ar[ddl]\ar[ddr]\ar[ddrr]&&\\
	\sCM {R}\simeq \rC_0(kQ),&Q\ar@{}[r]|=&\\
	&   & 1&2&\cdots&n-1&n } \]
Therefore, ${R}$ is of hereditary representation type.
\end{Thm}
\begin{proof}
	One can depict the extended numerical semigroup as follows.
	\[ S=
	\begin{tabular}{c|ccccc}
		&$1$&$\ast$&$\cdots$&$\ast$&$\la$\\
		\hline
		$0$&$\bullet$&&&&\\
		$1$&$\bullet$&$\bullet$&$\cdots$&$\bullet$&\\
		$2$&$\bullet$&$\bullet$&$\cdots$&$\bullet$&$\bullet$\\
		$3$&$\bullet$&$\bullet$&$\cdots$&$\bullet$&$\bullet$
	\end{tabular}
	\]
	We readily see that $S$ is a connected, twisted symmetric extended numerical semigroup ring with Frobenius number $1$, so that $R^\circ=k[S]$ is a $1$-dimensional positively graded reduced Gorenstein ring with $R_0=k$ and of $a$-invariant $1$. Then by \cite{BIY} we have a triangle equivalence
	\[ \sCM^\Z\!R^\circ \simeq \Db(\mod A) \]
	for $A=\begin{pmatrix}R^\circ_0&K^\circ_1\\0&K^\circ_0\end{pmatrix}$ and the graded total quotient ring $K^\circ=k[t^{\pm1}]\otimes k[L]$. Now by our assumption on the characteristic of $k$ we have $K^\circ_0=k[L]\simeq k\times\cdots\times k$ ($n$ copies) as $k$-algebras, thus $A=kQ$ for the $n$-subspace quiver $Q$ described above. By \cite[Theorem 6.5]{haI} we deduce a triangle equivalence
	\[ \sCM R^\circ\simeq \rC_0(kQ) \]
	for the ungraded Cohen-Macaulay modules, hence also $\sCM{R}\simeq\rC_0(kQ)$ for the complete case.
\end{proof}
In particular, ${R}$ is of finite representation type if $n\leq3$, and tame representation type when $n=4$.
\begin{Ex}\label{Klein}
	Let $L=\Z/2\Z\times\Z/2\Z=\{1,\la,\mu,\nu\}$ be the Klein $4$-group, and let $S\subset\NN\times L$ be the numerical subgroup generated by $(1,1)$, $(1,\la)$, and $(1,\mu)$, which is depicted as below.
	\[ S=
	\begin{tabular}{c|cccc}
		&$1$&$\la$&$\mu$&$\nu$\\
		\hline
		$0$&$\bullet$&&&\\
		$1$&$\bullet$&$\bullet$&$\bullet$&\\
		$2$&$\bullet$&$\bullet$&$\bullet$&$\bullet$\\
		$3$&$\bullet$&$\bullet$&$\bullet$&$\bullet$
	\end{tabular}
	\]
	Then it is connected, twisted symmetric, and has Frobenius number $1$. Over a field of characteristic $\neq2$ we obtain an equivalence
	\[ \sCM R\simeq \rC_0(kQ) \]
	for the extended Dynkin quiver $Q$ of type $\widetilde{D_4}$.
	In particular, $R$ is of tame representation type.
\end{Ex}

Note that even for the case $\ch k=2$ we do have triangle equivalences \cite{BIY,haI}, but $R$ is no longer of hereditary representation type.
\begin{Rem}
	In the setting of Example \ref{Klein}, suppose $\ch k=2$. The extended numerical semigroup ring $R$ is not reduced, in which case the category we shall study is those consisting of Cohen-Macaulay modules which are locally free on the punctured spectrum.
	\[ \begin{aligned}
		\sCM^\Z_0\!R&=\{ M\in \sCM^\Z\!R\mid M\otimes_RK\in\proj^\Z\!K\},\\
		\sCM_0\!R&=\{ M\in \sCM R \mid M\otimes_RK\in\proj K\}.
		\end{aligned} \]
	The algebra $K_0\simeq k[x,y]/(x^2,y^2)$ is non-semisimple, and $A=\begin{pmatrix}R_0&K_1\\0&K_0\end{pmatrix}$ is presented by the following quiver with relations.
	\[ \xymatrix{ \circ\ar[r]&\circ\ar@(r,u)[]_-x\ar@(r,d)[]^-y }, \quad xy=yx, x^2=y^2=0. \]
	By \cite[Theorem 6.5]{haI} we obtain a commutative diagram
	\[ \xymatrix@!R=3mm{
		\per A\ar[r]\ar@{-}[d]^-\rsimeq&\rC_0(A)\ar@{-}[d]^-\rsimeq\\
		\sCM_{0}^\Z\!R\ar[r]&\sCM_{0}\!R, } \]
	where $\rC_0(A)$ is the triangulated hull of $\per A/-\lotimes_ADA$.
\end{Rem}

\section{Some resolutions over polynomial rings}
This is a preparatory section for the computation we shall do in Section \ref{non-Gor}.
\medskip

Let $V$ be an $n$-dimensional vector space. For each $0\leq l\leq n$ and $m\geq0$ we have a canonical map
\[ \xymatrix@R=1mm{
	\al_{l,m}\colon\Wedge^{l}V\otimes D(\Sym^mV)\ar[r]&\Wedge^{l-1}V\otimes D(\Sym^{m-1}V) \\
	(x_1\wedge\cdots\wedge x_l)\otimes \varphi\ar@{|->}[r]& \Sum_{i=1}^l(-1)^{i-1}(x_1\wedge\cdots\wedge\widehat{x_i}\wedge\cdots\wedge x_l)\otimes\varphi x_i,} \]
where $x\in V$ acts on $\varphi\in D(\Sym^mV)$ by the dual of the multiplication map $x\colon\Sym^{m-1}V\to\Sym^mV$. For $m<0$ we regard $\Sym^mV=0$.
\begin{Lem}\label{tor}
For each $m\neq-n$, the sequence of vector spaces
\[ \xymatrix{ 0\ar[r]&\Wedge^{n}V\otimes D(\Sym^{m+n}V)\ar[r]^-{\al_{n,m+n}}&\cdots\ar[r]&V\otimes D(\Sym^{m+1}V)\ar[r]^-{\al_{1,m+1}}&D(\Sym^mV)\ar[r]&0 } \]
is exact.
\end{Lem}
\begin{proof}
	We view $V$ as a graded vector space concentrated in degree $1$, and the symmetric algebra $S=\Sym V$ as a graded ring. Let $E=E_S(k)$ be the injective hull of the simple $S$-module $k$ in $\Mod^\Z\!S$. We know that $E\simeq DS=\bigoplus_{m\in\Z}D(\Sym^mV)$, where $D$ is the graded dual $\bigoplus_{i\in\Z}M_i\mapsto \bigoplus_{i\in\Z}\Hom_k(M_{-i},k)$. Applying $-\otimes_SE$ to the Koszul complex of $S$ we have a sequence
	\[ \xymatrix{ 0\ar[r]&\Wedge^{n}V\otimes E\ar[r]&\cdots\ar[r]&V\otimes E\ar[r]&E\ar[r]&0 } \]
	which computes $\Tor_i^S(k,E)$. We have $k\lotimes_SE=k\lotimes_SDS=D\RHom_S(k,S)=D(k(n)[-n])=k(-n)[n]$, so that the above sequence is acyclic except at the leftmost term, where it has $1$-dimensional cohomology in degree $n$. We obtain the conclusion by taking the degree $-m$ part.
\end{proof}

Let $S=\Sym V$ be the symmetric algebra over $V$, so it is isomorphic to the polynomial ring in $n$ variables. Consider the Koszul complex
\[ \xymatrix{ 0\ar[r]&S\otimes\Wedge^nV\ar[r]^-{d_n}&S\otimes\Wedge^{n-1}V\ar[r]^-{d_{n-1}}&\cdots\ar[r]&S\otimes\Wedge^2V\ar[r]^-{d_2}&S\otimes V\ar[r]^-{d_1}&S } \]
of $S$ and put $\Diff^i=\Ker d_i$.
\begin{Lem}\label{dual}
There exists an exact sequence of $S$-modules
\[ \xymatrix{ 0\ar[r]&S\otimes\Wedge^nV\ar[r]&\Diff^{n-1}\otimes D(\Sym^{n-1}V)\ar[r]&\cdots\ar[r]&\Diff^{2}\otimes D(\Sym^{2}V)\ar[r]&\Diff\otimes DV\ar[r]&S}, \]
in which the cokernel of the last map is $k$.
\end{Lem}
\begin{proof}
	To save space we write $E^i$ for $D(\Sym^iV)$. Consider the diagram
	\newcommand{\DSym}[1]{E^{#1}}
	\newcommand{\DV}{E^1}
	\[ \xymatrix@C=2mm{
		&&&&&&&&S\otimes\Wedge^nV\otimes \DSym{n}\ar[d]\\
		&&&&&&S\otimes\Wedge^nV\otimes \DSym{n-1}\ar[rr]\ar[d]&&S\otimes\Wedge^{n-1}V\otimes \DSym{n-1}\ar[d]\\
		&&&&&\cdots&\cdots\ar[d]&&\cdots\ar[d]\\
		&&&&S\otimes\Wedge^nV\otimes\DSym{2}\ar[r]\ar[d]&\cdots\ar[r]&S\otimes\Wedge^3V\otimes\DSym{2}\ar[rr]\ar[d]&&S\otimes\Wedge^2V\otimes\DSym{2}\ar[d]\\
		&&S\otimes\Wedge^nV\otimes \DV\ar[rr]\ar[d]&&S\otimes\Wedge^{n-1}V\otimes \DV\ar[r]\ar[d]&\cdots\ar[r]&S\otimes\Wedge^2V\otimes \DV\ar[rr]\ar[d]&&S\otimes V\otimes \DV\ar[d]\\
		S\otimes\Wedge^nV\ar[rr]&&S\otimes\Wedge^{n-1}V\ar[rr]&&S\otimes\Wedge^{n-2}V\ar[r]&\cdots\ar[r]&S\otimes V\ar[rr]&&S.} \]
	Here, each row is obtained by applying $-\otimes E^i$ to a truncated Koszul complex, thus acyclic except at the rightmost position. Also, each column is obtained by applying $S\otimes-$ to the sequences in Lemma \ref{tor} (for $m=-n,-n+1,-n+2,\ldots,-1,0$), thus is exact except for the leftmost one. Moreover, it is easily verified that each square is commutative.
	
	Now, taking the images of the rightmost horizontal maps, we obtain a complex
	\[ \xymatrix{ 0\ar[r]&\Diff^{n-1}\otimes \DSym{n-1}\ar[r]&\cdots\ar[r]&\Diff^{2}\otimes \DSym{2}\ar[r]&\Diff\otimes\DSym{1}\ar[r]&\m\ar[r]&0}, \]
	where $\m$ is the maximal ideal of $S$ generated by $V$.
	Viewing the above diagram as a horizontal complex of vertical complexes, acyclicity of the middle vertical complexes yields that the above complex is quasi-isomorphic to $S\otimes\Wedge^nV$ shifted by $[n-1]$. This gives the desired assertion.
\end{proof}
\begin{Rem}
	An easy diagram chase shows that the first map $S\otimes\Wedge^nV\to\Diff^{n-1}\otimes D(\Sym^{n-1}V)$ is given as follows: take a basis $\{ m_i\}_i$ of $\Sym^{n-1}V$ and let $\{\mu_i\}_i$ be its dual basis. Then the map is given by
	\[ 1\otimes(x_1\wedge\cdots\wedge x_n)\mapsto\sum_i \pm m_i\otimes(x_1\wedge\cdots\wedge x_n)\otimes\mu_i \]
	under the canonical identification $S\otimes\Wedge^nV\xsimeq\Diff^{n-1}$.
\end{Rem}

\section{A non-Gorenstein example}\label{non-Gor}
So far we have been discussing {\it Gorenstein} rings of hereditary representation type. In this section, we give a {\it non-Gorenstein} example whose Cohen-Macaulay representations can be controlled by quiver representations, thus should be regarded as being hereditary representation type. 
\subsection{$2$-cluster tilting objects}
Let us collect some further results on cluster tilting theory, with an emphasis that ``$2$''-cluster tilting behave much nicely than general higher dimensional cluster tilting objects.

First of all, we note the following well-known lemma which will be frequently used.
\begin{Lem}\label{Om}
Let $\E$ be an exact category with enough projectives $\P$, $A,B\in\E$, and suppose that $\Ext^1_\E(A,P)=0$ for all $P\in\P$. Then there are natural isomorphisms
\[ \xymatrix@R=1mm{
	\sHom_\E(A,B)\ar[r]^-\simeq&\sHom_\E(\Om A,\Om B), \\
	\sHom_\E(\Om A,B)\ar[r]^-\simeq&\Ext^1_\E(A,B).} \]
\end{Lem}

The machinery of relating the representation theory of $R$ with that of finite dimensional algebra through a $2$-cluster tilting object is summarized in the following result. 
\begin{Prop}[{e.g.\,\cite[Theorem 3.2]{DL}}]\label{DL}
	Let $R$ be a complete Cohen-Macaulay local isolated singularity. Suppose that $\CM R$ contains a $2$-cluster tilting object $X$ and put $\Ga=\sEnd_R(X)$.
	\begin{enumerate}
		\item We have $\Ga=\sEnd_R(\Om X)$ and $\sHom_R(\Om X,X)=0$.
		\item The functor
		\[ \xymatrix{ \sHom_R(\Om X,-)\colon\sCM R/[X]\ar[r]&\mod\Ga } \]
		is an equivalence.
	\end{enumerate}
\end{Prop}

Let us note a complementary observation which is a non-Gorenstein analogue of \cite[Corollary 6.5(1)]{IYo}. We say that a functor $F$ from $\CM R$ to an abelian category $\A$ {\it preserves rigidity} if we have $\Ext^1_\A(FM,FM)=0$ whenever $\Ext^1_R(M,M)=0$.
\begin{Prop}\label{rigid}
In the setting of Proposition \ref{DL}, consider the functor $F=\sHom_R(\Om X,-)\colon\CM R\to\mod\Ga$. If $M,N\in\CM R$ satisfies $\Ext^1_R(M,N)=0$, then $\Hom_\Ga(FN,\tau FM)=0$. In particular, the functor $F$ preserves rigidity.
\end{Prop}
\begin{proof}
	Let $0\to M\to X^0\to X^1\to 0$ be the exact sequence with $X^0, X^1\in\add X$ and $X^0\to X^1$ being a radical map. Applying $F$ gives an exact sequence
	\[ \xymatrix{ \sHom_R(\Om X,\Om X^0)\ar[r]&\sHom_R(\Om X,\Om X^1)\ar[r]&\sHom_R(\Om X, M)\ar[r]& \sHom_R(\Om X,X^0) }. \]
	Since the last term is $0$ by Proposition \ref{DL}(1), it is the minimal projective presentation of $FM$ in $\mod\Ga$.
	Applying the Nakayama functor over $\Ga$ gives the exact sequence
	\[ \xymatrix{ 0\ar[r]& \tau FM\ar[r]&D\sHom_R(\Om X^0,\Om X)\ar[r]&D\sHom_R(\Om X^1,\Om X) } \]
	which computes $\tau FM$. Now, our assertion $\Hom_\Ga(FN,\tau FM)=0$ amounts to showing the map
	\[ \xymatrix{ \Hom_\Ga(FN,D\sHom_R(\Om X^0,\Om X))\ar[r]&\Hom_\Ga(FN,D\sHom_R(\Om X^1,\Om X)) } \]
	is injective. By duality, it is equivalent to the map
	\[ \xymatrix{ \sHom_R(\Om X^1,N)\ar[r]&\sHom_R(\Om X^0,N) }, \]
	or $\Ext^1_R(X^1,N)\to\Ext^1_R(X^0,N)$, being surjective. This follows from the assumption $\Ext^1_R(M,N)=0$.
\end{proof}
Note that this functor does not necessarily detect rigidity.
\subsection{Proof of Theorem \ref{CM}}\label{Proof}
Recall that $S=k[[x_0,x_1,y_0,y_1,y_2]]$. We give a complete grading on $S$ by $\deg x_i=(1,0)$, $\deg y_i=(0,1)$. The completely $\Z$-graded CM ring of our interest is
\[ R=S^{(1,1)}. \]
Also define the completely graded $R$-modules $M_i$ by $M_i=\prod_{j\in\Z}S_{(i,0)+(j,j)}$. Thus $M_i$ is generated in degree $0$ if $i\geq0$ and in degree $-i$ if $i<0$. By Corollary \ref{GP} (or \cite{St}), we know that the canonical module is isomorphic to $M_1$, and the modules $M_{-1}$, $M_0=R$, $M_1=\om_R$, and $M_2$ are CM.

Let us collect some basic information which will be used throughout this section.

\begin{Lem}\label{kos}
There exist exact sequences
\begin{enumerate}
	\item\label{xx} $\xymatrix{0\ar[r]&M_{i-2}\ar[r]&M_{i-1}^2\ar[r]&M_i\ar[r]&k[[y_0,y_1,y_2]]_{-i}\ar[r]&0}$ for all $i\in\Z$.
	\item\label{yyy} $\xymatrix{0\ar[r]&M_{i+3}\ar[r]&M_{i+2}^3\ar[r]&M_{i+1}^3\ar[r]&M_i\ar[r]&k[[x_0,x_1]]_i\ar[r]&0}$ for all $i\in\Z$.
\end{enumerate}
\end{Lem}
\begin{proof}
	We only prove (1). Consider the Koszul complex
	\[ \xymatrix{ 0\ar[r]&S(-2,0)\ar[r]&S(-1,0)^2\ar[r]&S\ar[r]&k[[y_0,y_1,y_2]]\ar[r]&0 } \]
	for $x_0,x_1\in S$. Taking its degree $\{(i,0)+(j,j)\mid j\in\Z\}$-part, we obtain the desired exact sequence.
\end{proof}
It follows we have an exact sequence $0\to M_{-1}\to R^2\to \om\to 0$, which shows $\Om\om=M_{-1}$. It will be helpful to visualize the category $\add\{M_i \mid i\in\Z\}\subset\mod R$ as follows.
\[ \xymatrix{
	\cdots\ar@2@/^7pt/[r]&M_{-2}\ar@2@/^7pt/[r]\ar@3@/^7pt/[l]&\Om\om\ar@2@/^7pt/[r]\ar@3@/^7pt/[l]&R\ar@2@/^7pt/[r]\ar@3@/^7pt/[l]&\om\ar@2@/^7pt/[r]\ar@3@/^7pt/[l]&M_{2}\ar@2@/^7pt/[r]\ar@3@/^7pt/[l]&\cdots\ar@3@/^7pt/[l] } \]
Here, the rightgoing double arrows are $x_0$ and $x_1$, and the leftgoing triple arrows are $y_0$, $y_1$, and $y_2$.

Note also that applying Lemma \ref{kos}(\ref{yyy}) for $i=-1$ we have a $4$-term exact sequence, in which the image of the middle map is $\Om^2\om$. Let us record this fact below.
\begin{Lem}\label{Om^2}
There exists an exact sequence below, where $\Om^2\om$ is the image of the middle map.
\[ \xymatrix@C=4mm@R=1.5mm{
	0\ar[rr]&&M_2\ar[rr]&&\om^3\ar[rr]\ar[dr]&&R^3\ar[rr]&&\Om\om\ar[rr]&&0\\
	&&&&&\Om^2\om\ar[ur]&&&&& } \]
\end{Lem}
We will often make use of the following useful result.
\begin{Prop}[{\cite[Proposition 2.5.1]{Iy07a}}]\label{depth}
Let $A$ be a commutative Cohen-Macaulay local isolated singularity of dimension $d\geq2$ and $M, N\in\CM A$. Then $\depth_A\Hom_A(M,N)=2+\max\{0\leq n<d-1\mid \Ext_A^i(M,N)=0 \text{ for all }0<i\leq n\}$.
\end{Prop}

Now let us give some computations of morphisms and extensions between some modules.
\begin{Lem}\label{ext}
\begin{enumerate}
	\item\label{sEnd} $\sEnd_R(\om)=\sEnd_R(\Om\om)=\sEnd_R(\Om^2\om)=k$.
	\item\label{om,R} $\Ext^i_R(\om,R)=0$ for $i=1,2,3$.
	\item\label{XX} $\Ext^1_R(X,X)=0$ for $X=R\oplus\om\oplus\Om^2\om$.
\end{enumerate}
\end{Lem}
\begin{proof}
	(\ref{om,R})  We have that $\Hom_R(\om,R)=\Hom_R(M_1,R)=M_{-1}$ is CM. Then Proposition \ref{depth} gives $\Ext^1_R(\om,R)=\Ext^2_R(\om,R)=0$. Similarly, since $\Hom_R(\Om\om,R)=\Hom_R(M_{-1},R)=M_1$ has depth $4$, we have $\Ext^2_R(\Om\om,R)=0$, hence $\Ext^3_R(\om,R)=0$.
	
	(\ref{sEnd})  In view of (\ref{om,R}) and Lemma \ref{Om}, the three endomorphism rings are isomorphic. Applying $\Hom_R(\om,-)$ to the exact sequence $0\to\Om\om\to R^2\to \om\to0$ we get $0\to\Hom_R(\om,\Om\om)\to\Hom_R(\om,R^2)\to\Hom_R(\om,\om)$ whose cokernel is $\sEnd_R(\om)$. Now the first three terms are isomorphic by Lemma \ref{cov} to $0\to M_{-2}\to M_{-1}^2\to R$, whose cokernel is $k$ by Lemma \ref{kos}(\ref{xx}), so we deduce $\sEnd_R(\om)=k$.
	
	(\ref{XX})  We have to show $\Ext_R^1(\om\oplus\Om^2\om,R\oplus\Om^2\om)=0$. By (1) we know $\Ext^1_R(\om\oplus\Om^2\om,R)=0$, so it remains to consider $\Ext^1_R(\om\oplus\Om^2\om,\Om^2\om)$. For $\Ext^1_R(\om,\Om^2\om)$, consider the exact sequence $0\to M_2\to \om^3\to\Om^2\om\to0$ from Lemma \ref{Om^2}. Applying $\Hom_R(\om,-)$ we see that $\Ext^1_R(\om,\Om^2\om)\xsimeq\Ext^2_R(M_2,\om)$, which vanishes by Proposition \ref{depth} since $\Hom_R(M_2,\om)\simeq M_{-1}$ has depth $4$. For $\Ext^1_R(\Om^2\om,\Om^2\om)$, this is isomorphic to $\Ext^3_R(\om,\Om^2\om)=\Ext^1_R(\om,\om)$ by (\ref{om,R}), thus $0$.
\end{proof}

\begin{Prop}\label{quiver}
The algebra $\Ga=\End_R(R\oplus\om\oplus\Om^2\om)$ is presented by the following quiver.
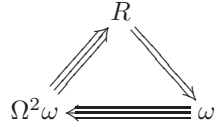
\[ \xymatrix@!C=3mm{
	&R\ar@2[dr]&\\
	\Om^2\om\ar@3[ur]&&\om\ar@3[ll]} \]
\end{Prop}	
\begin{proof}
	We put $X=R\oplus\om\oplus\Om^2\om$ and prove the following.
	\begin{enumerate}
	\item\label{sink at om} $f\colon R^2\to\om$ is the sink map at $\om$ in $\add X$.
	\item\label{sink at Om^2} $g\colon\om^3\to\Om^2\om$ is the sink map at $\Om^2\om$ in $\add X$.
	\item\label{source at Om^2} $h\colon\Om^2\om\to R^3$ is the source map at $\Om^2\om$ in $\add X$.
	\end{enumerate}
	These claims yield the quiver of $\End_R(X)$. Indeed, assertions (\ref{sink at Om^2}) and (\ref{source at Om^2}) exhaust the arrows adjacent to $\Om^2$, and (\ref{sink at om}) the arrows from $R$ to $\om$. It remains to show that there is no irreducible morphism in $\add X$ from $\om$ to $R$. Since $\Hom_R(\om,R)\simeq M_{-1}$ is generated by $y_0,y_1,y_2$ by Lemma \ref{kos}(\ref{yyy}), any one of these generators, hence any map $\om\to R$, factors through $\Om^2\om$ as can be seen from Lemma \ref{Om^2}.
	
	Now we turn to the proof of the claims.
		
	
	(\ref{sink at om})  Consider the exact sequence $0\to\Om\om\to R^2\xrightarrow{f}\om\to0$. Clearly any $R\to \om$ factors through $f$, and by Lemma \ref{ext}(\ref{sEnd}) so does any non-isomorphism $\om\to\om$. Also by Lemma \ref{ext}(\ref{om,R}) we see $\Ext^1_R(\Om^2\om,\Om\om)=\Ext^2_R(\Om\om,\Om\om)=\Ext^1_R(\Om\om,\om)=0$, so any $\Om^2\om\to\om$ also factors through $f$.
	
	(\ref{sink at Om^2})  Consider the exact sequence $0\to M_2\to \om^3\xrightarrow{g}\Om^2\om\to0$. Since $\Ext^1_R(R\oplus\om,M_2)=0$, any map $R\oplus\om\to\Om^2\om$ factors through $g$. Also we have that $\Ext^1_R(\Om^2\om,M_2)=\Ext^3_R(\om,M_2)$ is isomorphic by the exact sequence in Lemma \ref{Om^2} and Lemma \ref{ext}(\ref{om,R}) to $\Ext^1_R(\om,\Om\om)$, thus to $\sEnd_R(\om)=k$, which shows any non-isomorphism $\Om^2\om\to\Om^2\om$ factors through $g$.
	
	(\ref{source at Om^2})  This is similarly shown using the exact sequence $0\to \Om^2\om\xrightarrow{h}R^3\to\Om\om\to0$.
\end{proof}

We are now ready to give the essential part of our result. 
\begin{Thm}\label{main}
Let $R=k[[x_0,x_1]]\seg k[[y_0,y_1,y_2]]$. Then $X:=R\oplus \om\oplus \Om^2\om$ is a $2$-cluster tilting object.
\end{Thm}
\begin{proof}
	By Auslander correspondence \cite[Theorem 4.2.3]{Iy07b} (for $d=m=4$ and $n=2$) we have to show that $\Ga:=\End_R(X)$ satisfies the following.
	\begin{itemize}
	\item $\Ga$ is an isolated singularity, $\gd\Ga=4$, and $\depth_R\Ga\geq3$.
	\item $\pd_\Ga L\leq 3$ for all simple $\Ga$-modules $L$ with $Le=0$, where $e\in\Ga$ is the idempotent corresponding to $R$.
	\end{itemize}
	Clearly, $\Ga$ is an isolated singularity since $R$ is, and $\depth_R\Ga\geq3$ by Proposition \ref{depth} and Lemma \ref{ext}(\ref{XX}). To prove the remaining assertions, we compute the minimal projective resolution of the simple $\Ga$-modules, in other words, the sink sequences (in the sense of \cite[Definition 3.1]{Iy07a}) of $R$, $\om$, and $\Om^2\om$ in $\add X$.
	
	\medskip
	{\it Claim 1: The sink sequence at $\om$ is given by
		\[ \xymatrix{0\ar[r]&\Om^2\om\ar[r]&R^3\ar[r]&R^2\ar[r]&\om\ar[r]&0}. \]}
	By Proposition \ref{quiver} the sink map at $\om$ in $\CM R$ is $R^2\to\om$ which is given by multiplication by $x_0$ and $x_1$. By Lemma \ref{kos} it extends to an exact sequence $0\to\Om\om\to R^2\to\om\to0$. Connecting with the right half $0\to\Om^2\om\to R^3\to\Om\om\to0$ of the exact sequence in Lemma \ref{Om^2} we obtain the above sequence. To verify that this is indeed the sink sequence, we only have to show that $R^3\to\Om\om$ is the minimal right $(\add X)$-approximation, but this is immediate by $\Ext^1_R(X,\Om^2\om)=0$.
	
	\medskip
	{\it Claim 2: The sink sequence at $\Om^2\om$ is given by
		\[ \xymatrix{ 0\ar[r]&R\ar[r]&\om^2\ar[r]&\om^3\ar[r]&\Om^2\om\ar[r]&0}. \]}
	By Proposition \ref{quiver} the sink map at $\Om^2\om$ in $\CM R$ is $\om^3\to\Om^2\om$, which extends to the left half $0\to M_2\to \om^3\to\Om^2\om\to0$ of the exact sequence in Lemma \ref{Om^2}. We obtain the sequence by connecting with the exact sequence $0\to R\to\om^2\to M_2\to0$. As above, we see $\om^2\to M_2$ is a minimal right $(\add X)$-approximation by $\Ext_R^1(X,R)=0$.
	
	\medskip
	{\it Claim 3: The sink sequence at $R$ is given by
		\[ \xymatrix{ 0\ar[r]&\om\ar[r]&\om^6\ar[r]&\Om^2\om^2\oplus\om^6\ar[r]&\Om^2\om^3\ar[r]&R }. \]}
	By Proposition \ref{quiver} we know that the sink map is $\Om^2\om^3\to R$. We use Lemma \ref{dual} to compute its kernel. Applying Lemma \ref{dual} to the $3$-dimensional vector space $W$ generated by $\{y_0,y_1,y_2\}$ gives an exact sequence
	\[ \xymatrix{ 0\ar[r]&S^\pprime\otimes\Wedge^3W\ar[r]&S^\pprime\otimes\Wedge^3W\otimes D(\Sym^2W)\ar[r]&\Diff^1\otimes DW\ar[r]&S^\pprime\ar[r]&k\ar[r]&0 } \]
	with $S^\pprime=k[y_0,y_1,y_2]$ and $\Diff^1=\Ker(S^\pprime\otimes W\to S^\pprime)$. Tensoring $S^\prime=k[x_0,x_1]$ and completing, we have
	\[ \xymatrix{0\ar[r]&S\otimes\Wedge^3W\ar[r]&S\otimes\Wedge^3W\otimes D(\Sym^2W)\ar[r]&\widehat{(\Diff^1\otimes S^\prime)}\otimes DW\ar[r]&S\ar[r]&k[[x_0,x_1]]\ar[r]&0 } \]
	for our $S=k[[x_0,x_1,y_0,y_1,y_2]]$. Taking the degree $\{(j,j)\mid j\in\Z\}$-part we obtain an exact sequence
	\[ \xymatrix{0\ar[r]&M_3\ar[r]&\om^6\ar[r]^-a&\Om^2\om^3\ar[r]&R\ar[r]&k\ar[r]&0 } \]
	which extends the sink map $\Om^2\om^3\to R$.
	
	Now let $N:=\Im a$ so that we have an exact sequence
	\[ \xymatrix{0\ar[r]&M_3\ar[r]&\om^6\ar[r]^-b&N\ar[r]&0}. \]
	To check if $b$ is a right $(\add X)$-approximation, we compute $\Ext^1_R(X,M_3)$. Clearly $\Ext^1_R(R,M_3)=0$ and since $\Hom_R(\om,M_3)=M_2$ has depth $4$, we also have $\Ext^1_R(\om,M_3)=0$ by Proposition \ref{depth}. Note however, that $\Ext^1_R(\Om^2\om,M_3)\neq0$ so that $b$ is {\it not} a right $(\add X)$-approximation. In fact, we show $\Ext^1_R(\Om^2\om,M_3)=k^2$. Indeed, by Lemma \ref{kos}(\ref{xx}) there is an exact sequence $0\to \om\to M_2^2\to M_3\to 0$. Applying $\Hom_R(\Om^2\om,-)$ yields an isomorphism $\Ext^1_R(\Om^2\om,M_2^2)\xsimeq\Ext^1_R(\Om^2\om,M_3)$ which is induced by multiplication by $x_0$ and $x_1$. Also, the left half on the exact sequence
	\[ \xymatrix{ 0\ar[r]&M_2\ar[r]&\om^3\ar[r]&\Om^2\om\ar[r]& 0} \]
	from Lemma \ref{Om^2} yields $\Ext^1_R(\Om^2\om,M_2)=k$, and its non-zero element is presented by the same exact sequence. These discussions show that the sequence in the second row of the diagram below, which is obtained from the direct sum of two copies of the above exact sequence by pushing out along the multiplication $M_2^2\to M_3$ by $x_0$ and $x_1$, induces a projective cover $\Hom_R(\Om^2\om,\Om^2\om^2)\to\Ext^1_R(\Om^2\om,M_3)$ as $\End_R(\Om^2\om)$-modules.
	\[ \xymatrix{
		0\ar[r]&M_2^2\ar[r]\ar[d]&\om^6\ar[r]\ar[d]&\Om^2\om^2\ar[r]\ar@{=}[d]& 0\\
		0\ar[r]&M_3\ar[r]&E\ar[r]&\Om^2\om^2\ar[r]&0} \]
	
	Let us get back to the first exact sequence of the previous paragraph. Applying $\Hom_R(\Om^2\om,-)$ yields a surjection $\Hom_R(\Om^2\om,N)\to\Ext^1_R(\Om^2\om,M_3)$, so there exists a map $\Om^2\om^2\to N$ which yields a commutative diagram
	\[ \xymatrix{
		0\ar[r]&M_3\ar[r]\ar@{=}[d]&E\ar[r]\ar[d]&\Om^2\om^2\ar[r]\ar[d]&0\\
		0\ar[r]&M_3\ar[r]&\om^6\ar[r]^-b&N\ar[r]&0 } \]
	by pulling back along the rightmost vertical map. We have now obtained a right $(\add X)$-approximation $\om^6\oplus \Om^2\om^2\to N$ whose kernel is $E$.
	
	Finally we compute the right $(\add X)$-approximation of $E$. The last commutative diagram from the paragraph before last shows that there is a surjection $\om^6\to E$ whose kernel is the same as $M_2^2\to M_3$, which is just $M_1=\om$. Therefore we have an exact sequence
	\[ \xymatrix{ 0\ar[r]&\om\ar[r]&\om^6\ar[r]^-c&E\ar[r]&0}. \]
	Since $\om$ is an injective object in $\CM R$ we see that $c$ is indeed an approximation. Connecting the approximation sequences we have constructed, we obtain the claim.
	
	Now the sink sequences show that simple $\Ga$-modules at non-projective summands of $X$ have projective dimension $3$, and that at the projective summand has projective dimension $4$. This yields the desired Auslander-type condition of $\Ga$ and completes the proof of the theorem. 
\end{proof}


Consequently, we obtain that our $R=k[[x_0,x_1,y_0,y_1,y_2]]^{(1,1)}$ is of ``hereditary representation type''.
\begin{Cor}\label{kQ_3}
There exists an equivalence
\[ \xymatrix{\CM R/[X]\ar[r]^-\simeq&\mod kQ_3} \]
for the $3$-Kronecker quiver $Q_3\colon\xymatrix{\circ\ar@3[r]&\circ}$.
\end{Cor}
\begin{proof}
	By Theorem \ref{main} and Proposition \ref{DL}, we have an equivalence
	\[ \xymatrix{ \sHom_R(\Om X,-)\colon\sCM R/[X]\ar[r]^-\simeq&\mod\sEnd_R(X)}. \]
	By Proposition \ref{quiver} we see that $\sEnd_R(X)=kQ_3$.
\end{proof}

For later convenience we include the component of the Auslander-Reiten quiver of $\CM R$ containing $X$. One can find a shadow of the equivalence $\CM R/[X]\xsimeq\mod kQ_3$ by cutting off the summands of $X$ yields the preprojective and the preinjective component of $\mod kQ_3$.
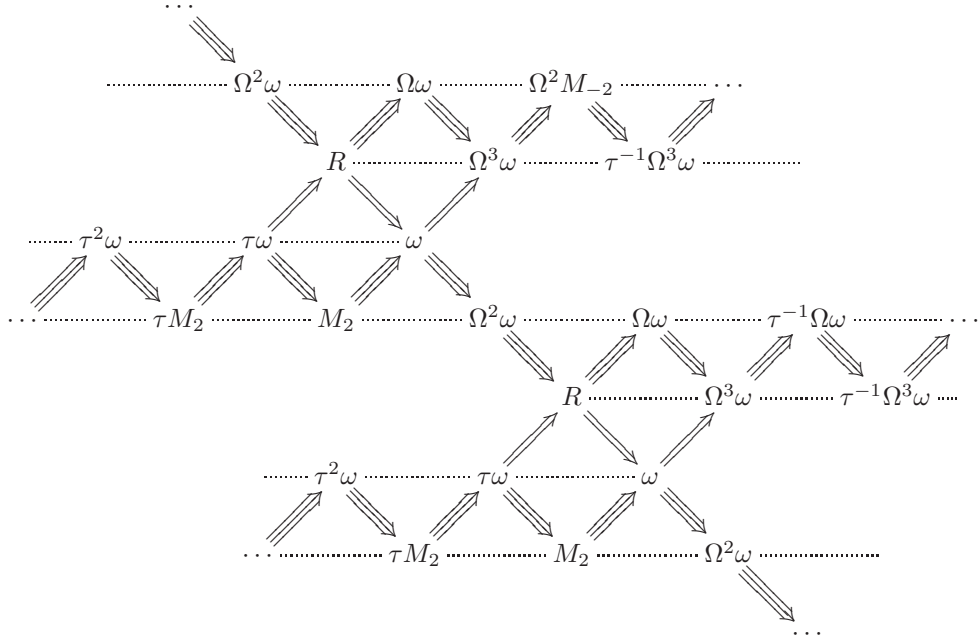
\[ \xymatrix@!R=2mm@!C=2mm{
	&&\cdots\ar@3[dr]&&&&&&&&&&\\
	&&&\Om^2\om\ar@3[dr]\ar@{..}[ll]&&\Om\om\ar@3[dr]\ar@{..}[ll]&&\Om^2M_{-2}\ar@3[dr]\ar@{..}[ll]&&\cdots\ar@{..}[ll]&&&\\
	&&&&R\ar@3[ur]\ar@2[dr]&&\Om^3\om\ar@3[ur]\ar@{..}[ll]&&\tau^{-1}\Om^3\om\ar@3[ur]\ar@{..}[ll]&&\ar@{..}[ll]&&\\
	&\tau^2\om\ar@3[dr]\ar@{..}[l]&&\tau\om\ar@3[dr]\ar@{..}[ll]\ar@2[ur]&&\om\ar@2[ur]\ar@3[dr]\ar@{..}[ll]&&&&&&&\\
	\cdots\ar@3[ur]&&\tau M_2\ar@3[ur]\ar@{..}[ll]&&M_2\ar@3[ur]\ar@{..}[ll]&&\Om^2\om\ar@3[dr]\ar@{..}[ll]&&\Om\om\ar@3[dr]\ar@{..}[ll]&&\tau^{-1}\Om\om\ar@3[dr]\ar@{..}[ll]&&\cdots\ar@{..}[ll]\\
	&&&&&&&R\ar@3[ur]\ar@2[dr]&&\Om^3\om\ar@3[ur]\ar@{..}[ll]&&\tau^{-1}\Om^3\om\ar@3[ur]\ar@{..}[ll]&\ar@{..}[l]\\
	&&&&\tau^2\om\ar@3[dr]\ar@{..}[l]&&\tau\om\ar@3[dr]\ar@{..}[ll]\ar@2[ur]&&\om\ar@2[ur]\ar@3[dr]\ar@{..}[ll]&&&&\\
	&&&\cdots\ar@3[ur]&&\tau M_2\ar@3[ur]\ar@{..}[ll]&&M_2\ar@3[ur]\ar@{..}[ll]&&\Om^2\om\ar@3[dr]\ar@{..}[ll]&&\ar@{..}[ll]&\\
	&&&&&&&&&&\cdots&&} \]
 
\subsection{Classification of rigid Cohen-Macaulay modules}
As in the previous subsection let
\[ R=k[[x_0,x_1]]\seg k[[y_0,y_1,y_2]] \]
be the complete Segre product of polynomial rings with standard gradings. We have seen that $X=R\oplus\om\oplus\Om^2\om$ is a $2$-cluster tilting object in $\CM R$ (Theorem \ref{main}) and it gives a functor
\[ \xymatrix{F\colon \CM R\ar@{->>}[r]&\CM R/[X]\ar[rr]^-{\sHom_R(\Om X,-)}&&\mod kQ_3 } \]
for the $3$-Kronecker quiver $\xymatrix{Q_3\colon\circ\ar@3[r]&\circ}$, in which the second one is an equivalence (Corollay \ref{kQ_3}). The aim of this subsection is to apply Corollay \ref{kQ_3} to give a complete classification of rigid Cohen-Macaulay modules over $R$.

Let $kQ_n$ be the Kronecker quiver with $n$ arrows. The Auslander-Reiten quiver of $\mod kQ_n$ contains connected components called preprojective component and the preinjective component, and we label the objects as $P_i$ and $I_i$ ($i\geq1$) as below.
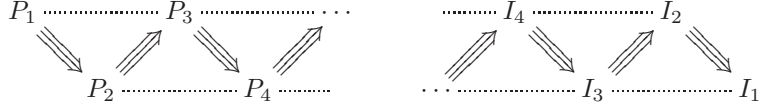
\[ 
\xymatrix@!R=2mm@!C=2mm{
	P_1\ar@3[dr]&&P_3\ar@3[dr]\ar@{..}[ll]&&\cdots\ar@{..}[ll]\\
	&P_2\ar@3[ur]&&P_4\ar@3[ur]\ar@{..}[ll]&\ar@{..}[l]}
\qquad
\xymatrix@!R=2mm@!C=2mm{
	&I_4\ar@3[dr]\ar@{..}[l]&&I_2\ar@3[dr]\ar@{..}[ll]&\\
	\cdots\ar@3[ur]&&I_3\ar@3[ur]\ar@{..}[ll]&&I_1\ar@{..}[ll]}
\]
As in Iyama--Yoshino's method, our classification theorem is based on such a result on quiver representations, more precisely Kac's theorem \cite{Kac}. Recall that a module is {\it basic} if its indecomposable summands are mutually non-isomorphic.
\begin{Prop}[{\cite{Kac}\cite[Section 7]{IYo}}]\label{kac}
Let $Q_n$ be the Kronecker quiver with $n$ arrows. Then the basic rigid $kQ_n$-modules are direct summands of $P_i\oplus P_{i+1}$ or $I_i\oplus I_{i+1}$ for $i\geq1$.
\end{Prop}
This leads to the following list of candidates of the rigid objects in $\CM R$. By abuse of notation we also denote by $P_i$ or $I_i$ the indecomposable Cohen-Macaulay $R$-module corresponding under $F$ to the $kQ_3$-modules in the preprojective or preinjective component. For example we have $P_1=\Om\om$, $P_2=\Om^3\om$, $P_3=\Om^2M_{-2}$, and $I_1=M_2$, $I_2=\tau\om$.
\begin{Lem}\label{sugu}
A basic object $M\in\CM R$ is rigid only if it is isomorphic in $\CM R/[X]$ to a direct summand of $P_i\oplus P_{i+1}$ or $I_i\oplus I_{i+1}$ for $i\geq1$.
\end{Lem}
\begin{proof}
	This is a straightforward consequence of Proposition \ref{rigid} and Proposition \ref{kac}.
\end{proof}

The main result of this subsection is the following complete classification of rigid Cohen-Macaulay modules over $R$. In fact, it turns out that only finitely many among the above candidates are rigid.
\begin{Thm}\label{classification}
A basic Cohen-Macaulay $R$-module is rigid if and only if it is a direct summand of one of the following.
\[ R\oplus\om\oplus\Om^2\om, \quad R\oplus\Om\om,\quad \om\oplus M_2 \]
Moreover, $R\oplus \om\oplus\Om^2\om$ is the unique basic $2$-cluster tilting object in $\CM R$.
\end{Thm}

For the proof we need some additional computations. We start with two of them which are already applications of Corollary \ref{kQ_3}.
\begin{Lem}\label{Ext-R-om}
Let $M\in\CM R$.
\begin{enumerate}
\item $\Ext^1_R(M,R)=0$ if and only if $M\in\add(X\oplus\Om\om)$.
\item $\Ext^1_R(\om,M)=0$ if and only if $M\in\add(X\oplus M_2)$.
\end{enumerate}
\end{Lem}
\begin{proof}
	We only prove (1) since (2) is dual. We know from Lemma \ref{ext} that $\Ext^1_R(X\oplus\Om\om,R)=0$. Suppose $\Ext^1_R(M,R)=0$. Then by Auslander-Reiten duality we have $\sHom_R(\Om^3\om,M)=0$, thus $\Hom_{kQ_3}(P_2,FM)=0$ under the functor $F\colon\CM R\to\mod kQ_3$. It follows that $FM$ must be the semisimple $kQ_3$-module $F(\Om\om)$ corresponding to the other vertex, thus $M\in\add(X\oplus \Om\om)$.
\end{proof}
\begin{Lem}\label{Ext-Om^2}
Let $M\in\CM R$.
\begin{enumerate}
\item $\Ext^1_R(M,\Om^2\om)=0$ if and only if $M\in\add(X\oplus M_2)$.
\item $\Ext^1_R(\Om^2\om,M)=0$ if and only if $M\in\add(X\oplus\Om\om)$.
\end{enumerate}
\end{Lem}
\begin{proof}
	(1)  Suppose $\Ext^1_R(M,\Om^2\om)=0$. By Auslander-Reiten duality we have $\sHom_R(\Om\om,M)=0$, thus $\Hom_{kQ_3}(P_1,FM)=0$ under the functor $F\colon\CM R\twoheadrightarrow\CM R/[X]\xrightarrow{\sHom_R(\Om X,-)}\mod kQ_3$. It follows that $FM$ must be the simple $kQ_3$-module $FM_2$ corresponding to the other vertex, thus $M\in\add(X\oplus M_2)$. To show the converse, we only have to verify $\Ext^1_R(M_2,\Om^2\om)=0$. Applying $\Hom_R(M_2,-)$ to the sequence in Lemma \ref{Om^2}, we have an exact sequence
	\[ \xymatrix@R=1.5mm{
		0\ar[r]&\Hom_R(M_2,M_2)\ar@{=}[d]\ar[r]&\Hom_R(M_2,\om^3)\ar@{=}[d]\ar[r]&\Hom_R(M_2,\Om^2\om)\ar[r]&\Ext^1_R(M_2,M_2)\ar@{=}[d]\\
		&R&M_{-1}^3&&0&} \]
	which shows $\Hom_R(M_2,\Om^2\om)$ must have depth $\geq3$. Then Proposition \ref{depth} gives the claim.
	
	(2)  This follows from (1) by applying the canonical dual. Note that $\Hom_R(\Om^2\om,\om)=\Om^2\om$ since the sequence in Lemma \ref{Om^2} is self-dual.
\end{proof}
We note another observation.
\begin{Lem}\label{PO}
Let $N\in\CM R$ and consider the push-out diagram
\[ \xymatrix{
	N\ar[r]^-j\ar[d]_-i&P\ar[d]\\
	I\ar[r]_-l&L } \]
with $i$ the inclusion into the injective hull $I\in\add\om$ and $j$ the minimal left $(\proj R)$-approximation. Then $l$ is a radical map and $L\in\add(X\oplus\Om\om)$.
\end{Lem}
\begin{proof}
	Since $i$ is minimal, the map $l$ is radical. The square being a push-out yields an exact sequence
	\[ \xymatrix{ 0\ar[r]& N\ar[r]&P\oplus I\ar[r]& L \ar[r]& 0 }. \]
	Applying $\Hom_R(-,\om)$ shows $L\in\CM R$ since $i$ is a left $(\add\om)$-approximation, and applying $\Hom_R(-,R)$ shows $\Ext^1_R(L,R)=0$ since $j$ is a left $(\proj R)$-approximation. Then Lemma \ref{Ext-R-om}(1) gives $L\in\add(X\oplus\Om\om)$.
\end{proof}
Now we are ready to prove the classification result.
\begin{proof}[Proof of Theorem \ref{classification}]
	It is easy to verify, using Lemma \ref{Ext-R-om}, that the listed modules are maximal rigid. We show the ``only if'' part. Let $M$ be a rigid module. By Lemma \ref{sugu}, it lies in $\add(X\oplus P_i\oplus P_{i+1})$ or in $\add(X\oplus I_i\oplus I_{i+1})$ for some $i\geq1$. If $M$ shares a summand with $X$ then Lemma \ref{Ext-R-om} and \ref{Ext-Om^2} immediately yield that $M$ must be a direct summand of one of the modules from the list.
	
	It remains to prove that $P_i$ and $I_i$ are {not} rigid when $i\geq2$. By the canonical dual it suffices to discuss $P_i$. For this we (re)consider the grading on $R$ as in the beginning of Section \ref{Proof}: we had graded modules $M_i=\prod_{j\in\Z}S_{(i,0)+(j,j)}$ over $R=M_0$. Then the Auslander--Reiten quiver of $\CM^\Z\!R$ becomes as below, where we gave the gradings on $P_i$'s so that $P_1=\Om M_1(1)$, $P_2=\Om^3M_1(3)$, and the rest are successors in the Auslander--Reiten quiver. We simiarly give the grading for the $I_i$'s. Note that we have $P_i=\Hom_R(I_i,M_1)$.
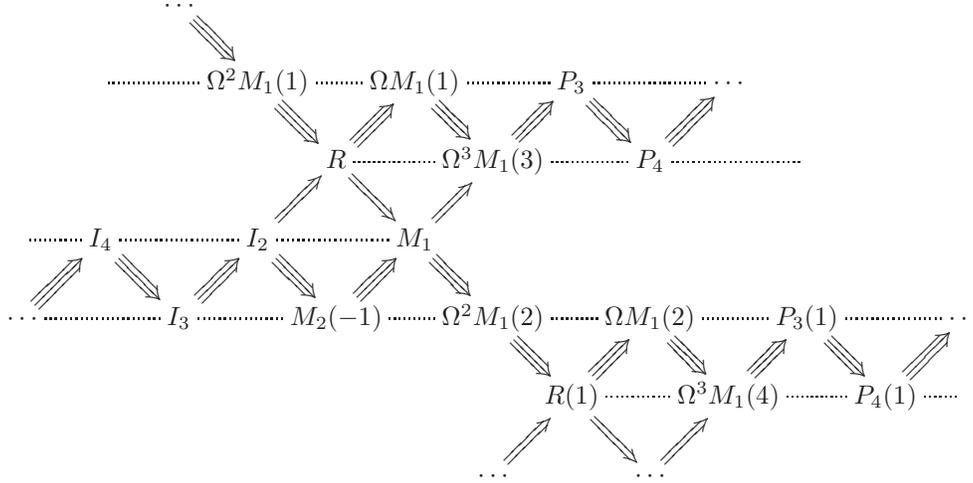
	\[ \xymatrix@!R=2mm@!C=2mm{
		&&\cdots\ar@3[dr]&&&&&&&&&&\\
		&&&\Om^2M_1(1)\ar@3[dr]\ar@{..}[ll]&&\Om M_1(1)\ar@3[dr]\ar@{..}[ll]&&P_3\ar@3[dr]\ar@{..}[ll]&&\cdots\ar@{..}[ll]&&&\\
		&&&&R\ar@3[ur]\ar@2[dr]&&\Om^3M_1(3)\ar@3[ur]\ar@{..}[ll]&&P_4\ar@3[ur]\ar@{..}[ll]&&\ar@{..}[ll]&&\\
		&I_4\ar@3[dr]\ar@{..}[l]&&I_2\ar@3[dr]\ar@{..}[ll]\ar@2[ur]&&M_1\ar@2[ur]\ar@3[dr]\ar@{..}[ll]&&&&&&&\\
		\cdots\ar@3[ur]&&I_3\ar@3[ur]\ar@{..}[ll]&&M_2(-1)\ar@3[ur]\ar@{..}[ll]&&\Om^2M_1(2)\ar@3[dr]\ar@{..}[ll]&&\Om M_1(2)\ar@3[dr]\ar@{..}[ll]&&P_3(1)\ar@3[dr]\ar@{..}[ll]&&\cdots\ar@{..}[ll]\\
		&&&&&&&R(1)\ar@3[ur]\ar@2[dr]&&\Om^3M_1(4)\ar@3[ur]\ar@{..}[ll]&&P_4(1)\ar@3[ur]\ar@{..}[ll]&\ar@{..}[l]\\
		&&&&&&\cdots\ar@2[ur]&&\cdots\ar@2[ur]&&&&} \]
	We proceed in several steps.
	
	{\it Step 1: Every $P_i$ is generated in degree $0$.}  Indeed, this is certainly the case for $P_1=\Om M_1(1)=M_{-1}(1)$ by the graded version of Lemma \ref{kos}(\ref{yyy}). Next, the almost split sequences $0\to R\to \Om M_1(1)\oplus M_1\to \Om^3M_1(3)\to 0$ shows that the rightmost term $P_2$ is generated in degree $0$. Inductively, the exact sequences $0\to P_{i-2}\to P_{i-1}^{\oplus3}\to P_i\to 0$ show that all $P_i$ must be generated in degree $0$.  
	
	{\it Step 2: Every $I_i$ is concentrated in degree $\geq1$, with non-zero degree $1$ part.}  The first part is proved in a similar way; It is the case for $I_1=M_2(-1)$, and the almost split sequences inductively yield the assertion. For the second part, let $d_i$ be the dimension of the degree $1$ part of $I_i$. Then we have $d_1=3$, and the almost split sequences yield $d_2=12$, and $d_i=3d_{i-1}-d_{i-2}$ for $i\geq3$. It easily follows that $d_i>0$.
	
	{\it Step 3: The construction.}	Now for $M=P_i$ with $i\geq2$ we construct a morphism $\tau^{-1}M\to M$ in $\CM R$ which does not factor through a projective module. Then Auslander--Reiten duality shall give $\Ext^1_R(M,M)\neq0$. In fact we give it as a morphism $\tau^{-1}M(-1)\to M$ in $\CM^\Z\!R$. Let $\tau^{-1}M(-1)\to I$ be the injective hull in $\CM^\Z\!R$. By Step 2, the projective cover $P\to I_i$ satisfies $P\in\add\{R(-i)\mid i\geq1\}$ and $R(-1)\in\add P$, thus dualizing by $\Hom_R(-,M_1(-1))$ we see that $I\in\add\{M_1(i)\mid i\geq0\}$ and $M_1\in\add I$. Now define the morphism $\tau^{-1}M\to M$ to be the compsite $\tau^{-1}M\hookrightarrow I\twoheadrightarrow M_1\to M$, where the second map is the projection onto one of the factors, and the third map is a non-zero composite $M_1\to P_2\to\cdots\to M$ of irreducible morphisms which exists by the following claim.
	
	{\it Step 4: There exists a composite $M_1\to P_2\to\cdots\to M$ of irreducible morphisms which is non-zero.}  Note first that any non-zero map $M_1\to P_2$ is injective. Indeed, this is the dual of a non-zero map $I_2\to R$, whose cokernel $C$ has dimension $<4$, thus the dual is injective by $\Hom_R(C,M_1)=0$. Next, for each $i\geq2$ we have a monomorphism $P_i\to P_{i+1}^{\oplus3}$ as the source map, thus their composition $P_2\to P_n^{\oplus 3^{n-2}}$ is injective. Therefore we have a non-zero map $M_1\to P_2\to M^{\oplus}$, thus one of the components is non-zero.
	
	{\it Step 5: The conclusion.}  We show that the morphism $\tau^{-1}M(-1)\to I\to M$ constructed in Step 3 does not factor through a projective module. If it does, it factors through the minimal left projective approximation $\tau^{-1}M(-1)\to P$, hence through the push-out $L$ of the diagram $I\leftarrow\tau^{-1}M(-1)\rightarrow P$, which lies in $\add(X\oplus\Om\om)$ (us an ungraded module) by Lemma \ref{PO}. Now we get a morphism $M_1\hookrightarrow I\to L\to M$ in $\CM^\Z\!R$ which is equal to the non-zero composite of irreducible morphisms in Step 4, and the second map $I\to L$, hence $M_1\to L$, is a radical map by Lemma \ref{PO}.
	\[ \xymatrix@R=5mm@!C=10mm{
		M_1\ar@{^(->}[r]\ar[d]&I\ar[r]&L\ar[r]&M\\
		\Om^2M_1(2)^{\oplus3}\ar[d]\ar[urr]\\
		R(1)^{\oplus9}\ar[d]\ar[uurrr]\\
		\Om M_1(2)^{\oplus27}\oplus M_1(1)^{\oplus18}\ar[uuurrr] } \]
	Now, it is easy to verify that the map $M_1\to\Om^2M_1(2)^{\oplus3}$ is the source map in $\add\{R(i)\oplus M_1(i)\oplus \Om^2M_1(i)\oplus \Om M_1(i)\mid i\in\Z\}$, so that the radical map $M_1\to L$ factors through $M_1\to \Om^2M_1(2)^{\oplus3}$. Then the map to $M$ factors through the source maps (in $\CM^\Z\!R$) as in the diagram above, thus the constructed map $M_1\to M$ in turn factors through $\add(M_1(1)\oplus\Om M_1(2))$. However, this is impossible for a non-zero map by $\Hom_R^\Z(M_1(1)\oplus \Om M_1(2),M)=0$; indeed, both $M_1(1)$ and $\Om M_1(2)=M_{-1}(2)$ are generated in degree $-1$ while $M$ is concentrated in degree $\geq0$.
\end{proof}
\end{document}